\documentclass[10pt]{amsart}
\usepackage{amsmath,amssymb,amsfonts,amsthm,amsopn}
\usepackage{latexsym,graphicx}
\usepackage{xcolor}
\usepackage{color, colortbl}
\usepackage{pb-diagram}
\usepackage[title]{appendix}
\usepackage{tikz-cd}
\usepackage{tikz}
\usepackage{mathtools}
\usepackage{enumerate}

\mathtoolsset{showonlyrefs}

\usepackage{multirow}

\newtheorem{theorem}{Theorem}[section]
\newtheorem{lemma}[theorem]{Lemma}
\newtheorem{corollary}[theorem]{Corollary}
\newtheorem{proposition}[theorem]{Proposition}

\newtheorem{remark}[theorem]{Remark}
\theoremstyle{definition}
\newtheorem{example}[theorem]{Example}

\newcommand{\beqa}{\begin{eqnarray*}}
\newcommand{\eeqa}{\end{eqnarray*}}

\DeclareMathOperator*{\Sp}{Sp}
\DeclareMathOperator*{\Mp}{Mp}
\DeclareMathOperator*{\Sym}{Sym}

\DeclareMathOperator*{\GL}{GL}

\newcommand{\field}[1]{\mathbb{#1}}
\newcommand{\bR}{\field{R}}        %  real numbers
\newcommand{\bZ}{\field{Z}}        %  whole numbers
\newcommand{\bC}{\field{C}}        %  complex numbers
\def\la{\lambda}

\def\cF{\mathcal{F}}              % Calligraphic Letters
\def\cS{\mathcal{S}}
\def\cD{\mathcal{D}}

\def\cJ{\mathcal{J}}

\def\cC{\mathcal{C}}

\def\rd{\bR^d}

\def\rdd{{\bR^{2d}}}

\def\lrd{L^2(\rd)}

\def\R{\right)}
\def\l{\langle}
\def\r{\rangle}
\def\<{\left<}
\def\>{\right>}

\def\mv1{M_v^1}

\def\mn{(m,n)}
\def\mn'{(m',n')}

\newcommand{\norm}[1]{\lVert#1\rVert}

\def\i{\infty}

\def\R{\mathbb{R}}
\def\Ren{\mathbb{R}^d}

\def\Sn2{S_{2}(L^{2}(\Ren))}
\def\S1{S_{1}(L^{2}(\Ren))}
\def\sig00{\sigma_{0,0}}
\def\la{\langle}
\def\ra{\rangle}

\begin{document}

\begin{abstract}
	We establish anisotropic uncertainty principles (UPs) for general metaplectic operators acting on
	$L^2(\mathbb{R}^d)$, including degenerate cases associated with symplectic matrices whose
	$B$-block has nontrivial kernel. In this setting, uncertainty phenomena are shown to be intrinsically
	directional and confined to an effective phase-space dimension given by $\mathrm{rank}(B)$.
	
	First, we prove sharp Heisenberg-Pauli-Weyl type inequalities involving only the directions
	corresponding to $\ker(B)^\perp$, with explicit lower bounds expressed in terms of geometric
	quantities associated with the underlying symplectic transformation. We also provide a complete
	characterization of all extremizers, which turn out to be partially Gaussian functions with free
	behavior along the null directions of $B$.
	
	Building on this framework, we extend the Beurling-H\"ormander theorem to the metaplectic
	setting, obtaining a precise polynomial-Gaussian structure for functions satisfying suitable
	exponential integrability conditions involving both $f$ and its metaplectic transform. Finally,
	we prove a Morgan-type (or Gel'fand--Shilov type) uncertainty principle for metaplectic operators,
	identifying a sharp threshold separating triviality from density of admissible functions and
	showing that this threshold is invariant under metaplectic transformations.
	
	Our results recover the classical Fourier case and free metaplectic transformations as special
	instances, and reveal the geometric and anisotropic nature of uncertainty principles in the
	presence of symplectic degeneracies.
\end{abstract}

\title[Anisotropic uncertainty principles for metaplectic operators]{Anisotropic uncertainty principles for metaplectic operators}

\author{Elena Cordero}
\address{Universit\`a di Torino, Dipartimento di Matematica, via Carlo Alberto 10, 10123 Torino, Italy}
\email{elena.cordero@unito.it}
\author{Gianluca Giacchi}
\address{Universit\`a della Svizzera Italiana, Faculty of Informatics, Via la Santa 1, 6962 Lugano, Switzerland}
\email{gianluca.giacchi@usi.ch}
\author{Edoardo Pucci}
\address{Universit\`a di Torino, Dipartimento di Matematica, via Carlo Alberto 10, 10123 Torino, Italy}
\email{edoardo.pucci@unito.it}

\thanks{}
%\date{}
\maketitle

\section{Introduction}
Uncertainty principles are a fundamental theme in harmonic analysis, expressing the intrinsic impossibility of simultaneously localizing a function and its Fourier transform in both space and frequency. Originating in the seminal works of Heisenberg, Weyl, Hardy, and Pauli, these principles have evolved into a broad and flexible framework encompassing sharp inequalities, anisotropic formulations, characterization of extremizers, and extensions adapted to non-Euclidean geometries, integral transforms, and dynamical settings. We refer to the surveys \cite{BDsurvey,FMBAMS21,BVJLMS2026} for a comprehensive overview, as well as to \cite{CF2015,CF2017,CEKPV2010,EKPV2006,FS,ST2024,WW21} for several influential contributions.

Beyond their intrinsic interest, uncertainty principles play a central role in partial differential equations, time-frequency analysis, and mathematical physics. In these contexts, they encode rigidity and uniqueness phenomena and often determine the qualitative structure of solutions to evolution equations, particularly those governed by quadratic Hamiltonians.

The  Heisenberg-Weyl uncertainty principle  can be stated in the following directional version (cf. \cite{BDJ} and \cite[Corollary 2.6]{FS}):

\begin{theorem}[Heisenberg’s Inequality]\label{thmHeis} Let $i = 1, \dots, d$ and $f \in L^2(\mathbb{R}^d)$. Then
	\begin{equation}\label{e00}
		\inf_{a \in \mathbb{R}} \int_{\mathbb{R}^d} (x_i - a)^2 |f(x)|^2 \,dx \cdot \inf_{b \in \mathbb{R}} \int_{\mathbb{R}^d} (\xi_i - b)^2 |\widehat{f}(\xi)|^2 \,d\xi \geq \frac{\|f\|^4_{2}}{16\pi^2}.
	\end{equation}
	Moreover, \eqref{e00} is an equality if and only if $f$ is of the form
	\begin{equation}
		f(x) = \Theta(x_1, \dots, x_{i-1}, x_{i+1}, \dots, x_d) e^{2\pi i b x_i} e^{-\gamma (x_i - a)^2},
	\end{equation}
	where $\Theta$ is a function in $L^2(\mathbb{R}^{d-1})$, $\gamma > 0$, and $a, b$ are real constants for which the two infima in \eqref{e00} are realized.
\end{theorem}
The general $d$-dimensional case  was proved by Folland and Sitaram in  \cite[Theorem 1.1]{FS}% CHECK: verify the numbering in [FS] (Theorem 1.1 there may be the one-dimensional statement)
 and states that
\begin{equation}\label{e0}
	\left(\int_{\bR^d}|x-a|^2 |f(x)|^2dx\right)\left(\int_{\bR^d}|\xi-b|^2|\widehat f(\xi)|^2d\xi\right)\geq\frac{d^2\norm{f}_2^4}{16\pi^2},
\end{equation}
for every $f\in\lrd$ and  $a,b\in\bR^d$. 
Equality holds if and only if $f(x)=\Theta e^{2\pi ib\cdot x}e^{-\gamma |x - a|^2}$ for some $\Theta \in\bC$ and $\gamma>0$.
Observe that this case can be inferred from the directional inequality in \eqref{e00}.

These results already indicate that uncertainty may naturally be investigated direction by direction, a perspective that becomes essential when dealing with transforms that do not treat all spatial and frequency variables uniformly.

Another form of UP is  due to Beurling and then proved by H\"ormander \cite{H}:

For $f \in L^{1}(\mathbb{R})$, if
\[
\int_{\mathbb{R}}\int_{\mathbb{R}}
|f(x)|\,|\widehat{f}(\xi)|\,e^{2\pi |x\cdot \xi|}\,dx\,d\xi < \infty,
\]
then $f=0$.
Folland and Sitaram extended the result to the $d>1$ setting in their survey \cite{FS} and, finally,  Bonami, Demange, and Jaming \cite{BDJ} weakened the assumptions so that nonzero solutions given by
Hermite functions were also possible:
\begin{theorem}\label{BeurlingFT}
	Let $f\in L^2(\bR^d)$ and $N\geq 0$. Then,
	\begin{equation}\label{beurlingcondition}
		\int_{\bR^{2d}}\frac{|f(x)\widehat f(\eta)|}{(1+|x|+|\eta|)^N}e^{2\pi|x\cdot \eta |}dxd\eta<\infty
	\end{equation}
	if and only if
	\begin{equation}
		f(x)=q(x)e^{-\pi Mx\cdot x}, \qquad x\in\bR^d,
	\end{equation}
	where $q$ is a polynomial of degree $<\frac{N-d}{2}$ and $M$ is a real positive definite symmetric matrix. In particular, $f=0$ if $N\leq d$.
\end{theorem}
%Our setting extends the previous result to any metaplectic operator, recapturing, in the special case $\widehat{S}=\cF$, Theorem \ref{BeurlingFT}.

Finally, we recall  Morgan's uncertainty principle.
One year after Hardy's result \cite{Hardy33}, Morgan \cite{Morgan} generalized it as follows
(in this statement only, the Fourier transform is normalized as $\widehat f(\xi)=\int_{\rd}f(x)e^{-ix\cdot\xi}dx$; with the normalization used in the rest of the paper the threshold below has to be rescaled accordingly, cf.\ Theorem~\ref{classicMorgan}):
% CHECK: with \widehat f(\xi)=\int f e^{-2\pi i x\xi} and p=q=2 the condition (2a)^{1/2}(2b)^{1/2}>1 is NOT sufficient (take f=e^{-\pi|x|^2}, a=b=3); hence the normalization remark above.
Let $f \in L^{1}(\mathbb{R}^{d})$ and assume that there exist constants $a,b>0$ and
exponents $p>2$ and $q\in(1,2)$ such that
\[
\frac{1}{p}+\frac{1}{q}=1
\]
and
\[
\int_{\mathbb{R}^{d}} |f(x)|\,e^{a|x|^{p}}\,dx < \infty,
\qquad
\int_{\mathbb{R}^{d}} |\widehat f(\xi)|\,e^{b|\xi|^{q}}\,d\xi < \infty.
\]
If
\[
(ap)^{1/p}(bq)^{1/q} > 1,
\]
then
\[
f \equiv 0.
\]
Hardy's UP \cite{Hardy33} corresponds to the limiting case $p=q=2$, i.e. Gaussian decay.

In \cite{BDJ} Bonami, Demange and Jaming generalized Morgan's UP as follows.
\begin{theorem}[Gel’fand-Shilov type UP]\label{classicMorgan}
	Let $1<p<2$ and let $q$ be the conjugate exponent. Assume that $f\in L^2(\rd)$ satisfies:\begin{equation*}
		\int_{\rd}|f(x)|e^{\frac{2\pi}{p}a^p|x_j|^p}dx<\i\quad \text{and}\quad \int_{\rd}|\widehat f(\xi)|e^{\frac{2\pi}{q}b^q|\xi_j|^q}d\xi<\i,
	\end{equation*}
	for some $j=1,\dots,d$ and for some $a,b>0$. \begin{itemize}
		\item [i)] If $ab>|\cos(p\pi/2)|^{1/p}$ then $f=0$.
		\item [ii)] If $ab<|\cos(p\pi/2)|^{1/p}$ there exists  a dense subset of functions in $\lrd$ which satisfy the above conditions for all $j=1,\dots,d.$
	\end{itemize}
\end{theorem}
%Theorem \ref{5.1} below extends the previous item to metaplectic operators. 
These results reveal a sharp threshold phenomenon and show that Gaussian-type behavior plays a distinguished role in uncertainty principles with exponential weights.

A natural and unifying framework for many of these results is provided by the metaplectic representation. Up to a phase factor, this representation associates to each symplectic matrix
\[
S=\begin{pmatrix} A & B \\ C & D \end{pmatrix} \in \mathrm{Sp}(d,\mathbb{R}), \qquad A,\, B,\, C,\, D\in\bR^{d\times d}
\]
a unitary operator \(\widehat S\) (metaplectic operator) acting on $L^2(\mathbb{R}^d)$. This class includes the Fourier transform, fractional and partial Fourier transforms, quadratic phase multipliers, and propagators of Schr\"{o}dinger equations with quadratic Hamiltonians; see \cite{Gos11} and Section~2 below. In this setting, classical uncertainty principles can be interpreted as statements on the joint localization of a function $f$ and its metaplectic image $\widehat{S}f$. We say that $S$ is {\em free} if $B\in\GL(d,\bR)$, in which case $\widehat S$ is called {\em quadratic Fourier transform}.

In recent years, there has been growing interest in uncertainty principles adapted to metaplectic and time-frequency representations. For free metaplectic operators, namely those associated with symplectic matrices whose $B$-block is invertible, several analogues of classical Heisenberg, Hardy, Beurling, and Morgan-type uncertainty principles have been established, see for instance \cite{CGP2024,Liu,Z1,ZH}. In this regime, all spatial and frequency directions are involved, although the singular spectrum of $B$ may still produce anisotropy.
Complementing the Heisenberg-type uncertainty principles for general metaplectic operators developed in \cite{CGP2024,Lerner,LPDM2026}, our approach highlights the geometry of the degenerate case by identifying $\ker(B)^\perp$ as the effective space of uncertainty and deriving sharp anisotropic inequalities along these directions.

More recently, Gröchenig and Shafkulovska initiated a systematic study of uncertainty principles for metaplectic time-frequency representations, beyond the free case. In \cite{GrochenigShafkulovska2025}, the authors proved a Benedicks-type uncertainty principle for metaplectic representations, highlighting new rigidity phenomena and showing how the geometry of the underlying symplectic transformation influences phase-space localization. This line of research was further developed in the recent preprint \cite{ShafkulovskaGrochenig2025b}, where additional uncertainty principles of Heisenberg and Beurling type are established in the metaplectic time-frequency setting.

The presence of degeneracies in the symplectic matrix, and in particular of a nontrivial kernel of the block $B$, leads to a fundamentally different picture. In such cases, the metaplectic operator does not mix all spatial and frequency variables uniformly, and uncertainty phenomena become intrinsically directional. Certain directions in phase space exhibit genuine uncertainty constraints, while others remain essentially free. 
This aspect is overshadowed in \cite{ShafkulovskaGrochenig2025b}, where, for instance, strict global decay or integrability assumptions are imposed in order to obtain uncertainty principles. In fact, it is generally sufficient to require decay or integrability only along the specific directions determined by $\ker(B)^\perp$. This is the approach followed in \cite{HUP}.

The main novelty of the present paper is the extension of classical uncertainty principles to arbitrary metaplectic operators, including degenerate cases in which $\ker B \neq \{0\}$, by means of the technique introduced in \cite{CGP2024} and generalized in \cite{HUP}. Our approach relies on a careful geometric decomposition of phase space adapted to $\ker B$ and its orthogonal complement. As a consequence, uncertainty is confined to an effective phase-space dimension determined by $\mathrm{rank}(B)$, while the remaining directions exhibit free behavior.

More precisely, for metaplectic operators with $\mathrm{rank}(B)=r$, $1 \le r \le d$, we establish sharp anisotropic Heisenberg-Pauli-Weyl type inequalities involving only the $r$ effective directions of uncertainty, with explicit constants expressed in terms of geometric quantities associated with the underlying symplectic matrix. We also provide a complete characterization of extremizers, which turn out to be partially Gaussian functions with free behavior along the null directions of $B$.

Building on this framework, we prove a metaplectic version of the Beurling-Hörmander
theorem \ref{BeurlingFT}, extending the classical characterization of functions satisfying exponential
integrability conditions involving both \(f\) and its metaplectic transform \(\widehat S f\), cf. Theorem \ref{Thrm4.1} below.
In the free case  our result recovers the classical theorem for the Fourier transform,
while in the general case it yields a precise polynomial-Gaussian structure along the
effective directions.

Finally, we generalize Theorem  \ref{classicMorgan} and  establish a Morgan-type (or Gel'fand-Shilov type) UP for
metaplectic operators, see Theorem \ref{5.1} below.
As in the classical theory, we identify a sharp threshold separating triviality from
density of admissible functions, and show that this threshold is invariant under
metaplectic transformations.

These results recover the classical Fourier case and the quadratic Fourier transforms as special instances, and reveal the geometric and anisotropic nature of uncertainty principles in the presence of symplectic degeneracies.

Further connections with evolution equations and dispersive phenomena arise in the study of Schr\"odinger propagators with quadratic Hamiltonians, see the Examples \ref{ex4.8} and \ref{ex4.10} below for the cases of the free particle and the quantum harmonic oscillator. In this direction, uncertainty principles have been shown to encode strong rigidity and uniqueness properties for solutions of dispersive PDEs; see, for instance, \cite{CF2015,CF2017,HUP}. These results emphasize the deep interplay between symplectic geometry, metaplectic operators, and time-frequency analysis.

From a broader perspective, the conceptual and geometric significance of the uncertainty principle in harmonic analysis and mathematical physics has been highlighted in \cite{WW21}, while recent advances on subcritical and refined uncertainty principles are discussed in \cite{ST2024}. Finally, uncertainty principles for generalized Fourier-type and free metaplectic transforms have been actively investigated in recent years; see \cite{Liu,Z1,ZH} and the references therein. To sum up, the present work complements and extends these contributions by providing a unified and geometrically transparent treatment of anisotropic uncertainty principles for arbitrary metaplectic operators, including degenerate cases.

\noindent
\textbf{Structure of the paper.}
In Section~2 we collect preliminary material on symplectic geometry, metaplectic operators,
and geometric volume factors.
Section~3 is devoted to anisotropic Heisenberg-type inequalities, with sharp constants
and complete characterization of extremizers.
In Section~4 we prove a metaplectic version of Beurling's theorem 
 and, finally,  a Morgan-type uncertainty principle.
 
Several examples, including the Fourier transform, partial Fourier transforms, and
quadratic phase multipliers, are discussed throughout the paper to illustrate the results.

%{Nell'introduzione, per semplificare gli enunciati, si potrebbero togliere le espressioni esplicite delle costanti e rimandare alla section 3 per l'espressione degli optimizer. Secondo me, comunque, gli optimizer si possono semplificare, ci penso su.}

\section{Preliminaries}
For $x,y\in\rd$, we denote by $x\cdot y$ the standard inner product of $\rd$. We denote by $\cS(\rd)$ the Schwartz class of rapidly decreasing functions. Its topological dual $\cS'(\rd)$ is the space of tempered distributions.

We use the notation in \cite{MO2002}. For a $d \times d$ matrix $A$ and a linear subspace $L$ of $\mathbb{R}^d$ with $\dim L = \ell$, $q_L(A)$ denotes the $\ell$-dimensional volume of 
\[
X = \left\{ x \in \mathbb{R}^d \;\middle|\; x = \alpha_1 A e_1 + \cdots + \alpha_\ell A e_\ell,\ 0 \leq \alpha_i \leq 1,\ i = 1, \dots, \ell \right\}
\]
spanned by the vectors $Ae_1, \dots, Ae_\ell$, with $e_1, \dots, e_\ell$  any orthonormal basis of $L$.

If $\dim A(L) = \dim L = \ell$, then the $\ell$-dimensional volume of $X$ is positive; otherwise, this volume is zero. The number $q_L(A)$ can be associated with a matrix volume as follows: we collect the vectors $e_1, \dots, e_\ell$ as columns into the $d \times \ell$ matrix $E$ ($E = [e_1, \dots, e_\ell]$). Assuming that $\dim A(L) = \ell$, the matrix $AE$ has full column rank and
\begin{equation}\label{qL}
q_L(A) = \mathrm{vol}(AE) = \sqrt{\det(E^T A^T A E)}.
\end{equation}
If $\ell=0$ we set $q_L(A):=1$.
Observe that, if $L = \mathbb{R}^d$ and $A$ is nonsingular, then  $q_L(A) = |\det A|$.

 The following relation holds for every nonnegative Borel measurable function $\varphi$:
\begin{equation}\label{CV}
\int_L \varphi(Ax)\, dx = \frac{1}{q_L(A)} \int_{A(L)} \varphi(x)\, dx \quad (\dim A(L) = \dim L).
\end{equation}

\subsection{The symplectic group}
Let
\begin{equation}\label{defJ}
    J=\begin{pmatrix}
        0_d & I_d\\
        -I_d & 0_d
    \end{pmatrix}
\end{equation}
be the standard symplectic form of $\rdd$. A matrix
\begin{equation}\label{blockS}
	S=\begin{pmatrix}
		A & B\\
		C & D
	\end{pmatrix}\in\bR^{2d\times2d}
\end{equation}
is {\em symplectic} if {$S^T JS=J$}. Alternatively, 
\begin{align}
    & A^T C=C^T A,\\
    & B^T D=D^T B,\\
    \label{SpRel3}
    & A^T D-C^T B=I_d.
\end{align}
The symplectic group $\Sp(d,\bR)$ is generated by $J$ in \eqref{defJ} and by the subgroups
\begin{equation}\label{defVQ}
    V_Q=\begin{pmatrix}
        I_d & 0_d\\
        Q & I_d
    \end{pmatrix}, \qquad Q\in\bR^{d\times d}, \, Q=Q^T
\end{equation}
and
\begin{equation}\label{defDE}
    \cD_E=\begin{pmatrix}
        E^{-1} & 0_d\\
        0_d & E^T
    \end{pmatrix}, \qquad E\in\GL(d,\bR).
\end{equation}
We set
\begin{equation}\label{muS}
	\mu_S=\sqrt{\frac{1}{q_{R(B)^\perp}(A^T)\sigma(B)}},
\end{equation}
where $\sigma(B)$ is the product of the non-zero singular values of $B$, whereas $q_{R(B)^\perp}(A^T)$ is the volume of the {parallelotope} generated by $A^Tv_1,\ldots,A^Tv_{d-r}$, being $v_1,\ldots,v_{d-r}$ an orthonormal basis of $R(B)^\perp$.

Observe that $A^T: R(B)^\perp	\to \ker B$ is an isomorphism, as recalled below.

\begin{lemma}\label{ele1} Consider the symplectic matrix with block decomposition in \eqref{blockS}. Then\\
	(i) $D^TA:\ker(B)\to D^TA(\ker(B))$ is an isomorphism. \\
	(ii) $D^T: A(\ker(B))\to D^TA(\ker(B))$ is an isomorphism. \\
	(iii) $A^T: R(B)^\perp	\to \ker B$ is an isomorphism.
\end{lemma}
\begin{proof}
Item $(i)$  is shown in Subsection 4.3 of \cite{HUP}. For item $(ii)$ it is enough to show that  $D^T: A(\ker(B))\to D^TA(\ker(B))$ is injective. By contradiction, assume there exist $x_1\not=x_2$ such that $D^Tx_1=D^Tx_2$. Then there exist $y_1,y_2\in\ker(B)$ such that $Ay_1=x_1\not=x_2=Ay_2$. By Item $(i)$
$D^TAy_1\not=D^TAy_2$, which is a contradiction. Item $(iii)$ is contained in Corollary B.2 and Table B.1 of \cite{HUP}.
\end{proof}

{
For later use, we denote by
\begin{equation}\label{defRestrictedD}
    \mathcal D:=D^T\big|_{A(\ker B)}:A(\ker B)\longrightarrow D^TA(\ker B)
\end{equation}
the isomorphism in Lemma~\ref{ele1}~$(ii)$.
}

A symplectic matrix $S$ with blocks \eqref{blockS} is orthogonal if and only if $C=-B$ and $D=A$. In this case, equation \eqref{SpRel3} reads as $A^T A+B^T B=I_d$, whence $A^T A x=x$ for every $x\in\ker(B)$. The same argument for $S^{-1}=S^T$ yields $AA^T x=x$ for every $x\in\ker(B^T)=R(B)^\perp$. In conclusion:
\begin{lemma}\label{lemmaOrthogonal}
    Let $S\in\Sp(d,\bR)$ be orthogonal with blocks \eqref{blockS}. Then, $A^T A|_{\ker(B)}=\mathrm{id}$ and $AA^T|_{R(B)^\perp}=\mathrm{id}$.
\end{lemma}

\subsection{The metaplectic group}
Let us define metaplectic operators as in \cite{MO2002}. For standard presentations, we refer to \cite{corderobook,deGosSGQM,deGosQHM}. If $S\in\Sp(d,\bR)$ has blocks as in \eqref{blockS}, we may consider the integral operator
\begin{equation}\label{defS1}
    \widehat Sf(x)=\mu_Se^{i\pi DC^T x\cdot x}\int_{\ker(B)^\perp}f(y+D^T x)e^{i\pi B^+Ay\cdot y}e^{2\pi iC^T x\cdot y}dy,
\end{equation}
initially defined for $f\in\cS(\rd)$. If $B=0_d$, the expression above must be interpreted as
\begin{equation}\label{defS2}
    \widehat Sf(x)=\det(A)^{-1/2}e^{i\pi CA^{-1}x\cdot x}f(A^{-1}x).
\end{equation}
Notice the ambiguity of sign due to the square-root. 
An operator in the form \eqref{defS1} or \eqref{defS2} is called {\em metaplectic}. Metaplectic operators map continuously $\cS(\rd)$ to itself, and they are homeomorphisms. They extend to unitary operators on $L^2(\rd)$ and to homeomorphisms of $\cS'(\rd)$, by setting
\begin{equation}
    \la\widehat Sf,g\ra =\l f,\widehat S^{-1}g\r, \qquad f,g\in \cS(\rd).
\end{equation}
The metaplectic operator associated to $S\in\Sp(d,\bR)$ is defined up to a phase factor, i.e. an unimodular complex constant. The group $\{\widehat S:S\in\Sp(d,\bR)\}$ has a subgroup consisting of precisely two operators for each $S\in\Sp(d,\bR)$. This subgroup is the {\em metaplectic group} and it is denoted by $\Mp(d,\bR)$. The projection $\pi^{Mp}:\widehat S\in\Mp(d,\bR)\mapsto S\in\Sp(d,\bR)$ is a group homomorphism with kernel $\{\pm\mathrm{id}\}$. 

%We recall Corollary $3.2$ in \cite{HUP}.
%\begin{corollary}\label{cor44}
%	Let $f\in L^2(\rd)$  and $\xi\in\rd$ having decomposition %$\xi=\xi_1+\xi_2$ be as above. Then, the following integral %representation holds:
%	\begin{align}
%		\label{f2}
%		\widehat Sf(\xi)&{=}\mu_S e^{i\pi (DB^+ %\xi_1\cdot\xi_1+DC^T\xi_2\cdot\xi_2)}\int_{\ker(B)^\perp}f(t+D^T\xi_2)e^{%i\pi B^+At\cdot t}e^{-2\pi i (B^+\xi_1-C^T\xi_2)\cdot t}dt.
%	\end{align}
%	The formula above is understood as the equality of two $L^2(\R^d)$ %functions.
%\end{corollary}
%{Here we may cite the work on kernels of metaplectic operators with Luigi, where we used the formula with $R(B)^\perp$ instead of $A(\ker(B))$.}
\begin{example}\label{exMetap}
	(a) The Fourier transform $\cF$, defined for every $f\in \cS(\rd)$ by
	\begin{equation}
		\cF f(\xi)=\widehat f(\xi)=\int_{\rd}f(x)e^{-2\pi i\xi\cdot x}dx, \qquad \xi\in\rd,
	\end{equation}
	is a metaplectic operator. Its projection is $\pi^{Mp}(\cF)=J$, defined in \eqref{defJ}.\\
	(b) If $Q\in\Sym(d,\bR)$,  \begin{equation}\label{defPhi}
			\Phi_Q(t)=e^{i\pi Qt\cdot t}
		\end{equation}
  is the corresponding chirp, and the operator
	\begin{equation}
		\mathfrak{p}_Qf(t)=\Phi_Q(t)f(t), \qquad f\in L^2(\rd),
	\end{equation}
	is metaplectic, with projection $\pi^{Mp}(\mathfrak{p}_Q)=V_Q$, defined in \eqref{defVQ}. \\
	(c) If $E\in\GL(d,\bR)$, the rescaling operator:
	\begin{equation}
		\mathfrak{T}_Ef(t)=|\det(E)|^{1/2}f(Et), \qquad f\in L^2(\rd),
	\end{equation}
	is metaplectic, with projection $\pi^{Mp}(\mathfrak{T}_E)=\cD_E$, defined in \eqref{defDE}.\\
	(d) If $P\in\Sym(d,\bR)$, the multiplier operator:
	\begin{equation}
		\mathfrak{m}_Pf=\cF^{-1}(\Phi_{-P} \widehat f), \qquad f\in L^2(\rd)
	\end{equation}
	is metaplectic, with $\pi^{Mp}(\mathfrak{m}_P)=V_{P}^T$, see \eqref{defVQ}.
\end{example}

\subsection{Anisotropic tensor products}

We shall use the following density result.

\begin{lemma}
Let $V_1,V_2$ be linear subspaces of $\mathbb{R}^d$, possibly of
dimension zero, such that
\[
\mathbb{R}^d=V_1\oplus V_2.
\]
Then the subspace
\[
\mathfrak{D}:=\operatorname{span}\bigl\{
f\in L^2(\mathbb{R}^d):
f(x_1+x_2)=h(x_1)g(x_2),\;
h\in L^2(V_1),\;g\in L^2(V_2)
\bigr\}
\]
is dense in $L^2(\mathbb{R}^d)$.
\end{lemma}

\begin{proof}
Let $T:V_1\times V_2\to\mathbb{R}^d$ be the linear isomorphism
$T(x_1,x_2)=x_1+x_2$, and let $J_T>0$ be the modulus of its
determinant with respect to orthonormal coordinates on $V_1$ and $V_2$.
The map
\[
U:L^2(\mathbb{R}^d)\longrightarrow L^2(V_1\times V_2),
\qquad
(Uf)(x_1,x_2)=J_T^{1/2}f(x_1+x_2),
\]
is unitary by the linear change-of-variables formula.
The algebraic tensor product
$L^2(V_1)\otimes_{\mathrm{alg}}L^2(V_2)$
is dense in $L^2(V_1\times V_2)$.
Its inverse image under $U$ consists of finite linear combinations
of functions of the form $J_T^{-1/2}h(x_1)g(x_2)$ and is therefore
contained in $\mathfrak{D}$.
Hence $\mathfrak{D}$ is dense in $L^2(\mathbb{R}^d)$.
\end{proof}

In particular, we will use the following Corollary:

\begin{corollary}\label{densita}
Let $V_1,V_2$ be linear subspaces of $\mathbb{R}^d$, possibly of
dimension zero, such that
\[
\mathbb{R}^d=V_1\oplus V_2.
\]
Let $H_1,H_2$ be dense subsets of $L^2(V_1)$ and $L^2(V_2)$,
respectively. Then the subspace
\[
\mathfrak{D}:=\operatorname{span}\bigl\{
f\in L^2(\mathbb{R}^d):
f(x_1+x_2)=h(x_1)g(x_2),\;
h\in H_1,\;g\in H_2
\bigr\}
\]
is dense in $L^2(\mathbb{R}^d)$.
\end{corollary}

\section{Uncertainty principle of Heisenberg-Pauli-Weyl}
Let $\widehat S$ be a metaplectic operator on $L^2(\rd)$ with projection \eqref{blockS} and $B \neq 0_d$.
The case $B=0_d$ does not provide uncertainty, as already observed in \cite[Proposition 1.4]{HUP}:
\begin{proposition}\label{propAssumption}
Let $\widehat{S}$ be a metaplectic operator with $B=0_d$. Then there exists
$f \in L^2(\mathbb{R}^d)\setminus\{0\}$ such that $f$ and $\widehat{S}f$ have compact support.
\end{proposition}
In view of Proposition \ref{propAssumption}, for the rest of this work, we may assume $B\neq0_d$ without loss of generality.\\

Let $V:\bR^r\to\ker(B)^\perp$ be  an arbitrary  parametrization which maps the canonical basis of $\bR^r$ to an orthonormal basis of $\ker(B)^\perp$. In this case, the Moore-Penrose inverse of $V$ coincides with $V^T$, and $V^TV=I_{ r}$. We refer to Figure \ref{fig:1} for a visual representation of these mappings.
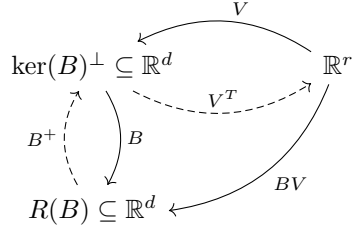
\begin{figure}
\begin{tikzcd}[row sep=huge, column sep=huge]
	\ker(B)^\perp\subseteq\rd\arrow[r,"V^T", bend right=30, dashrightarrow] \arrow[r,"V", bend left=30,leftarrow]  \arrow[d, bend left = 30,"B"] & \bR^r  \arrow[dl, "BV",bend left=30]  \\
	R(B)\subseteq\rd \arrow[u,bend left=30,"B^+",dashrightarrow]
\end{tikzcd}
\caption{A schematic representation of the action of the block $B$, its Moore-Penrose inverse $B^+$, and the parametrization $V$ of $\ker(B)^\perp$, chosen so that $V^+=V^T$.}
\label{fig:1}
\end{figure}
In what follows, we will use the following decompositions of $\rd$ into direct sums, proved in \cite[Lemmas 3.4 and 3.6]{HUP} (see also Lemma \ref{ele1}):
\begin{equation}\label{D1}
	x=x_1+x_2\in\ker(B)^\perp\oplus D^TA(\ker(B)),
\end{equation}
and
\begin{equation}\label{D2}
	\xi=\xi_1+\xi_2\in R(B)\oplus A(\ker(B)). 
\end{equation}
If we assume that $\dim (\ker(B)^\perp)=r$ then by \eqref{D1} we have $\dim (D^TA(\ker(B)))=d-r$.
We recall \cite[Corollary 3.2]{HUP}:
\begin{corollary}\label{cor44}
	Let $f\in L^2(\rd)$  and $\xi=\xi_1+\xi_2$, where $\xi_1\in R(B)$ and $\xi_2\in A(\ker(B))$. Then, the following integral representation holds
	\begin{align}
		\label{f2}
		\widehat Sf(\xi)&{=}\mu_S e^{i\pi (DB^+ \xi_1\cdot\xi_1+DC^T\xi_2\cdot\xi_2)}\int_{\ker(B)^\perp}f(t+D^T\xi_2)e^{i\pi B^+At\cdot t}e^{-2\pi i (B^+\xi_1-C^T\xi_2)\cdot t}dt.
	\end{align}
	The formula above is understood as the equality of two $L^2(\R^d)$ functions.
\end{corollary}
As a consequence \cite[Corollary 3.7]{HUP}
\begin{corollary}\label{cor45}
	Under the notation of Corollary \ref{cor44},  we can write
	\[
		|\widehat Sf(\xi)|=\mu_S |\widehat g_{\xi_2}(V^TB^+\xi_1)|,
	\]
	where $\xi=\xi_1+\xi_2$, $\xi_1\in R(B)$ and $\xi_2\in A(\ker(B))$, and, {for almost every $\xi_2\in A(\ker(B))$},
	$\widehat g_{\xi_2}$ is the Fourier transform on $\bR^r$ of the function
	\begin{equation}\label{gu}
		g_{\xi_2}(u)=f(Vu+D^T\xi_2)e^{i\pi (V^TB^+AVu\cdot u+2V^TC^T\xi_2\cdot u)}, \qquad u\in\bR^r.
% CHECK: the sign of the term 2V^TC^T\xi_2\cdot u is + (it comes from e^{+2\pi i C^T\xi_2\cdot t} in \eqref{f2}); [HUP, (25)] has -, due to a typo (B^+\xi_1+C^T\xi_2) in the first line of the proof of [HUP, Cor. 3.7]. Verified numerically.
	\end{equation}
\end{corollary}

Observe that, by a standard Fubini argument, {$g_{\xi_2}\in L^2(\bR^r)$ for almost every $\xi_2\in A(\ker(B))$}.
Using the decomposition \eqref{D2} and observing that  $\dim(R(B))=r$ we infer $\dim  A(\ker(B))=d-r$. 

%{{Let $\widetilde V:\bR^{d-r}\to A(\ker(B))$   an arbitrary  parametrization which maps the canonical basis of $\bR^{d-r}$ to an orthonormal basis of $A(\ker(B))$. Again, $\widetilde V^T\widetilde V=I_{d- r}$ and the Moore-Penrose inverse of $\widetilde V$ is $\widetilde V^T$.}}
%\begin{equation}\label{gunew}
%	g(u,\eta):=f(Vu+D^T\widetilde V\eta )e^{i\pi (V^TB^+AVu\cdot u-2V^TC^T\widetilde V\eta\cdot u)}, \qquad u\in\bR^r,\eta\in \bR^{d-r}.
%\end{equation}
\begin{theorem}\label{dirHeis}
	Let $\widehat S$ be a metaplectic operator on $L^2(\rd)$ with projection \eqref{blockS}. Assume that $1\leq r= \mathrm{rank}(B)\leq d$ and fix an orthonormal basis $\{v_1,\dots,v_r\}$ of $\ker(B)^\perp$. Assume the decompositions in \eqref{D1} and \eqref{D2}. Then, for every  $\alpha,\beta\in\bR$, $f\in\lrd$, 
			{
			\begin{equation}\label{e5}
			\left(\int_{\bR^d}|x_1^{(j)}-\alpha|^2 |f(x)|^2dx\right)^{1/2}\left(\int_{\bR^d}|\xi_1^{(j)}-\beta|^2|\widehat Sf(\xi)|^2d\xi\right)^{1/2}\geq \frac{1}{4\pi}\norm{f}_2^2,
		\end{equation}
		}
		 $j=1,\dots,r$, where $x_1^{(j)}$ is the $j$-th component of $x_1$ with respect to the orthonormal basis $\{v_1,\dots,v_r\}$ of $\ker(B)^\perp$ and $\xi_1^{(j)}$ is the $j$-th component of $\xi_1$ with respect to the basis $\{ Bv_1,\dots,Bv_r\}$ of $R(B)$. Furthermore, equality holds in \eqref{e5} if and only if
            \begin{equation}\label{flimitej}
f(x_1+x_2)
= \Theta \bigl(x_2;\,x_1^{(1)},\dots,x_1^{(j-1)},x_1^{(j+1)},\dots,x_1^{(r)}\bigr)\,
  e^{-\gamma_j(x_1^{(j)}-\alpha)^2}
  e^{i\pi\Phi(x_1,x_2)},
\end{equation}
where, if $\mathcal D$ is defined as in \eqref{defRestrictedD},
\[
\Phi(x_1,x_2)
= 2 \beta x_1^{(j)}
-  \bigl(B^+ A x_1 \cdot x_1 + 2 C^{T} \mathcal D^{-1} x_2 \cdot x_1\bigr),
\]
and \eqref{flimitej} holds for almost every
$(x_1,x_2)\in\ker(B)^\perp\times D^TA(\ker B)$, for some $\Theta \in L^2(D^TA(\ker(B))\times\bR^{r-1})$ and some $\gamma_j>0$ independent of $x_2$. (If $r=1$, $\Theta$ depends only on $x_2$; if $\ker B=\{0\}$, $\Theta$ depends only on $x_1^{(i)}$, $i\neq j$.)
\end{theorem}
\begin{proof}
If at least one of the integrals is infinite we are done. Assume now that both the integrals in \eqref{e5} are finite.
	First, we treat the more difficult scenario in which $\ker B\ne \{0\}$. Choose $V:\bR^r\to \ker(B)^\perp$ to be the isometric parametrization determined by \[
    Ve_j=v_j,\quad j=1,\dots,r,
    \] where $\{e_1,\dots,e_r\}$ is the canonical basis of $\bR^r$. Hence $V^TV=I_r$ and $V^Tv_j=e_j$ for every $j=1,\dots,r$. {For almost every $\xi_2\in A(\ker(B))$}, we consider $g_{\xi_2}$ in \eqref{gu} and the change of variables $u=V^Tt$ in \eqref{f2}, which yields
	\[
|\widehat Sf(\xi)|=\mu_S |\widehat g_{\xi_2}(V^TB^+\xi_1)|,
\]
see Corollary \ref{cor45}.
Since $\xi_1\in R(B)$ and $B^+:R(B)\to \ker(B)^\perp$, we write $\eta=V^TB^+\xi_1\in \bR^r$ so that $\xi_1=BV\eta$ and 
	\begin{equation}\label{relSfhatg}
	|\widehat g_{\xi_2}(\eta)|=\frac{1}{\mu_S}|\widehat Sf(BV\eta+\xi_2)|, \qquad {\text{for a.e. }\eta\in\bR^r}.
\end{equation}
For every $\alpha,\beta\in\bR$, using the directional Heisenberg uncertainty principle for $g_{\xi_2}$ in \eqref{e00} (or, equivalently, \eqref{e0} in dimension $d=1$), we obtain, {for almost every $\xi_2\in A(\ker(B))$},
	\begin{equation}\label{Heisen-g-1}
		\left(\int_{\bR^r}|u_j-\alpha|^2 | g_{\xi_2}(u)|^2du\right)^{1/2}\left(\int_{\bR^r}|\eta_j-\beta|^2|\widehat  g_{\xi_2}(\eta)|^2d\eta\right)^{1/2}\geq\,\,\frac{\norm{ g_{\xi_2}}_2^2}{4\pi}.
	\end{equation}
	By \eqref{gu},
	\begin{equation}
		\norm{ g_{\xi_2}}_2^2=\int_{\bR^r}|f(Vu+D^T\xi_2)|^2du=\int_{\ker(B)^\perp}|f(x_1+D^T\xi_2)|^2dx_1.
	\end{equation}

{For almost every $\xi_2\in A(\ker(B))$}, using Lemma \ref{ele1} we write $D^T\xi_2=D^TAy_2$, $y_2\in \ker B$, so that, by \eqref{qL},
	\begin{align*}
		\int_{A(\ker(B))}\norm{ g_{\xi_2}}_2^2d\xi_2&=\int_{A(\ker(B))}\int_{\ker(B)^\perp}|f(x_1+D^T\xi_2)|^2dx_1d\xi_2\\
		&=q_{\ker B}(A)\int_{\ker(B)}\int_{\ker(B)^\perp}|f(x_1+D^TAy_2)|^2dx_1dy_2.\\
		%&=\int_{\bR^r}\int_{\bR^{d-r}}|f(Vu+\widetilde V\eta)|^2du d\eta.
	\end{align*}
{
The decomposition
\begin{equation*}
    \rd= \ker(B)^{\perp}\oplus D^TA(\ker B)
\end{equation*}
allows us to write every $x\in \rd$ as
\begin{equation*}
    x=x_1+D^TAy_2,\qquad x_1\in \ker(B)^{\perp},\quad y_2\in\ker(B).
\end{equation*}
Let the columns of $E_1$ and $E_2$ be orthonormal bases of $\ker(B)^\perp$ and $\ker(B)$, respectively, and set $E=(E_1\mid E_2)$.  
With respect to these orthonormal parametrizations, the matrix of
\begin{equation*}
    \Psi(x_1,y_2)=x_1+D^TAy_2
\end{equation*}
is $N=(E_1\mid D^TAE_2)$. Since $E$ is orthogonal and $D^TA=I_d+B^TC$, while $BE_2=0$, we have
\begin{align*}
 |\det N| &=|\det(E^TN)|=\left|\det
 \begin{pmatrix}
 	I_r&E_1^TD^TAE_2\\
 	0_{(d-r)\times r}&E_2^TD^TAE_2
 \end{pmatrix}\right|=\left|\det
 \begin{pmatrix}
 	I_r&E_1^TD^TAE_2\\
 	0_{(d-r)\times r}&I_{d-r}
 \end{pmatrix}\right|=1.
\end{align*}
Thus the Jacobian of $\Psi$ is $1$, and the change-of-variables formula gives
\begin{align*}
   \int_{\ker(B)}\int_{\ker(B)^\perp}|f(x_1+D^TAy_2)|^2dx_1dy_2=\int_{\rd}|f(x)|^2dx=\norm{f}_2^2.
\end{align*}
Consequently,
\begin{equation}\label{e3}
    \int_{A(\ker B)}\norm{g_{\xi_2}}_2^2d\xi_2=q_{\ker B}(A)\norm{f}_2^2.
\end{equation}
}
	Next, we focus on the first integral in \eqref{Heisen-g-1}.
	\begin{align}
		\int_{\bR^r}|u_j-\alpha|^2 |{g_{\xi_2}(u)}|^2du&=\int_{\ker(B)^\perp}|x_1^{(j)}-\alpha|^2|f(x_1+D^T\xi_2)|^2dx_1,
	\end{align}
	where we used that {$x_1^{(j)}= x_1\cdot v_j=x_1\cdot Ve_j=V^Tx_1\cdot e_j= u\cdot e_j=u_j$}. Consequently,
	{
	\begin{equation}\label{e1}
		\int_{A(\ker(B))}\int_{\bR^r}|u_j-\alpha|^2 | g_{\xi_2}(u)|^2dud\xi_2=q_{\ker B}(A)\int_{\rd}|x_1^{(j)}-\alpha|^2|f(x)|^2dx.
	\end{equation}
	}
	Concerning the second integral in \eqref{Heisen-g-1}, setting $\xi_1=BV\eta\in R(B)$, by using the expression:
    \begin{align}
		BV\eta=\xi_1=\displaystyle\sum_{j=1}^r\xi_1^{(j)}Bv_j=\displaystyle\sum_{j=1}^r\xi_1^{(j)}BVe_j=BV(\displaystyle\sum_{j=1}^r\xi_1^{(j)}e_j),
	\end{align}
    and that $BV: \bR^r\to R(B)$ is an isomorphism, so $\eta_j=\xi_1^{(j)}$, we get:
	\begin{align*}
		\int_{\bR^r}|\eta_j-\beta|^2|\widehat  g_{\xi_2}(\eta)|^2d\eta&=\frac{1}{\mu_S^2}\int_{\bR^r}|\eta_j-\beta|^2|\widehat Sf(BV\eta+\xi_2)|^2d\eta\\
		&=\frac{1}{\sigma(B)\mu_S^2}\int_{R(B)}|\xi_1^{(j)}-\beta|^2|\widehat Sf(\xi_1+\xi_2)|^2d\xi_1,
		%&=\frac{C_B}{\mu_S^2}\int_{R(B)}|V^TB^+(\xi_1-b)|^2|\widehat Sf(\xi_1+\xi_2)|^2d\xi_1,\\
	\end{align*}
   In the last equality we used the fact that $V$ is an isometry, so the product of non-zero singular values of $B$ is the $r$-volume of the {parallelotope} generated by the image of the canonical basis of $\bR^r$ {through} $BV$.
   
	 Therefore,  using the explicit expression  of the constant $\mu_S$ in \eqref{muS} we obtain \begin{equation*}
	    \int_{\bR^r}|\eta_j-\beta|^2|\widehat  g_{\xi_2}(\eta)|^2d\eta=q_{R(B)^{\perp}}(A^T)\int_{R(B)}|\xi_1^{(j)}-\beta|^2|\widehat Sf(\xi_1+\xi_2)|^2d\xi_1,
	\end{equation*}
	so that
	\begin{align*}
		\int_{A(\ker(B))}&\int_{\bR^r}|\eta_j-\beta|^2|\widehat  g_{\xi_2}(\eta)|^2d\eta d\xi_2\\
        &=q_{R(B)^{\perp}}(A^T)\int_{A(\ker(B))}\int_{R(B)}|\xi_1^{(j)}-\beta|^2|\widehat Sf(\xi_1+\xi_2)|^2d\xi_1d\xi_2\\
		&=q_{R(B)^{\perp}}(A^T)q_{\ker B}(A)\int_{\ker B}\int_{R(B)}|\xi_1^{(j)}-\beta|^2|\widehat Sf(\xi_1+A\xi_2')|^2d\xi_1d\xi_2'.
	\end{align*}
    Now, we recall that $A^T:R(B)^{\perp}\to \ker B$ by Lemma \ref{ele1} is an isomorphism, so we write:
\begin{align*}
		q_{R(B)^{\perp}}(A^T)q_{\ker B}(A)&\int_{\ker B}\int_{R(B)}|\xi_1^{(j)}-\beta|^2|\widehat Sf(\xi_1+A\xi_2')|^2d\xi_1d\xi_2'\\
        &=q_{R(B)^{\perp}}(A^T)q_{\ker B}(A)\int_{A^T(R(B)^{\perp})}\int_{R(B)}|\xi_1^{(j)}-\beta|^2|\widehat Sf(\xi_1+A\xi_2')|^2d\xi_1d\xi_2'\\
        &=q_{\ker B}(A)q_{R(B)^{\perp}}(A^T)^2\int_{R(B)^{\perp}}\int_{R(B)}|\xi_1^{(j)}-\beta|^2|\widehat Sf(\xi_1+AA^T\xi_2'')|^2d\xi_1d\xi_2''.
	\end{align*}
{
To compute the last change of variables, let the columns of $F_1$ and $F_2$ be orthonormal bases of $R(B)$ and $R(B)^\perp$, respectively, and set $F=(F_1\mid F_2)$. The matrix of
\begin{equation*}
    (\xi_1,\xi_2'')\longmapsto \xi_1+AA^T\xi_2''
\end{equation*}
with respect to these parametrizations is $N=(F_1\mid AA^TF_2)$.
Therefore,
\begin{align*}
 |\det N|=|\det(F^TN)|=\left|\det
 \begin{pmatrix}
 	I_r&F_1^TAA^TF_2\\
 	0_{(d-r)\times r}&F_2^TAA^TF_2
 \end{pmatrix}\right|=\det(F_2^TAA^TF_2)=q_{R(B)^\perp}(A^T)^2,
\end{align*}
where the last identity follows from \eqref{qL}. Hence
\begin{equation}
    \int_{R(B)^{\perp}}\int_{R(B)}|\xi_1^{(j)}-\beta|^2|\widehat Sf(\xi_1+AA^T\xi_2'')|^2d\xi_1d\xi_2''=\frac{1}{q_{R(B)^{\perp}}(A^T)^2}\int_{\rd}|\xi_1^{(j)}-\beta|^2|\widehat S f(\xi)|^2d\xi.
\end{equation}
Consequently,
\begin{equation}
    \int_{A(\ker(B))}\int_{\bR^r}|\eta_j-\beta|^2|\widehat  g_{\xi_2}(\eta)|^2d\eta d\xi_2=q_{\ker B}(A)\int_{\rd}|\xi_1^{(j)}-\beta|^2|\widehat S f(\xi)|^2d\xi.\label{e2}
\end{equation}
}
Integrating the left-hand side of \eqref{Heisen-g-1} on $A(\ker(B))$ and using H\"{o}lder's inequality with $p=q=2$
	\begin{equation}\label{holder1}
    \begin{split}
	&\int_{A(\ker(B))}\left(\int_{\bR^r}|u_j-\alpha|^2 | g_{\xi_2}(u)|^2du\right)^{1/2}\left(\int_{\bR^r}|\eta_j-\beta|^2|\widehat  g_{\xi_2}(\eta)|^2d\eta\right)^{1/2} d\xi_2\\
	&\leq 	\left(\int_{A(\ker(B))}\int_{\bR^r}|u_j-\alpha|^2 |g_{\xi_2}(u)|^2du d\xi_2\right)^{1/2}\left(\int_{A(\ker(B))}\int_{\bR^r}|\eta_j-\beta|^2|\widehat g_{\xi_2}(\eta)|^2d\eta d\xi_2\right)^{1/2}\\
	&={q_{\ker B}(A)}\left(\int_{\bR^d}|x_1^{(j)}-\alpha|^2 |f(x)|^2dx\right)^{1/2}\left(\int_{\bR^d}|\xi_1^{(j)}-\beta|^2|\widehat Sf(\xi)|^2d\xi\right)^{1/2}.
    \end{split}
\end{equation}
{Finally, integrating both sides of \eqref{Heisen-g-1} over $A(\ker(B))$, using the upper bound above for the left-hand side and \eqref{e3} for the right-hand side, the common factor $q_{\ker B}(A)$ cancels and yields \eqref{e5}.}\\
Let us now characterize the equality in \eqref{e5}, assuming $f\neq0$. Since \eqref{Heisen-g-1} holds for a.e. $\xi_2$ and \eqref{holder1} is an inequality, equality in \eqref{e5} holds if and only if (a) equality holds in \eqref{Heisen-g-1} for a.e. $\xi_2\in A(\ker(B))$, and (b) equality holds in the H\"older inequality \eqref{holder1}.
By Theorem \ref{thmHeis} (applied in $\bR^r$; the case of $g_{\xi_2}=0$ being trivial), condition (a) holds if and only if, for a.e. $\xi_2$,
\begin{equation}\label{glimitej}
g_{\xi_2}(u)=\Theta_{\xi_2}(u_1,\dots,u_{j-1},u_{j+1},\dots,u_r) e^{2\pi i \beta u_j}e^{-\gamma_j(\xi_2)(u_j-\alpha)^2},\qquad \text{a.e. }u\in\bR^r,
\end{equation}
for some $\Theta_{\xi_2}\in L^2(\bR^{r-1})$ and $\gamma_j(\xi_2)>0$. For $g_{\xi_2}$ as in \eqref{glimitej}, by separation of variables and the one-dimensional Gaussian integrals (see also the proof of Theorem \ref{Teo2}),
\begin{equation*}
\int_{\bR^r}|u_j-\alpha|^2 | g_{\xi_2}(u)|^2du=\frac{\norm{g_{\xi_2}}_2^2}{4\gamma_j(\xi_2)},\qquad \int_{\bR^r}|\eta_j-\beta|^2|\widehat  g_{\xi_2}(\eta)|^2d\eta=\frac{\gamma_j(\xi_2)}{4\pi^2}\norm{g_{\xi_2}}_2^2 .
\end{equation*}
Condition (b) holds if and only if there exists a constant $c_j\ge0$ such that
\begin{equation*}
    \left(\int_{\bR^r}|u_j-\alpha|^2 | g_{\xi_2}(u)|^2du\right)^{1/2}=c_j\cdot \left(\int_{\bR^r}|\eta_j-\beta|^2|\widehat  g_{\xi_2}(\eta)|^2d\eta\right)^{1/2},\quad \text{for a.e. }\xi_2\in A(\ker(B)),
\end{equation*}
that is, by the identities above, if and only if $\pi/\gamma_j(\xi_2)=c_j$ for a.e. $\xi_2$ such that $g_{\xi_2}\neq0$. Hence $\gamma_j(\xi_2)=\gamma_j$ is a constant, and conversely every family \eqref{glimitej} with $\gamma_j$ independent of $\xi_2$ satisfies (a) and (b).
By \eqref{gu} and the change of variable $x_1=Vu$, \eqref{glimitej} is equivalent to
\begin{align*}
    f(x_1+D^T\xi_2)&=\Theta_{\xi_2}(x_1^{(1)},\dots,x_1^{(j-1)},x_1^{(j+1)},\dots,x_1^{(r)})e^{2\pi i \beta x_1^{(j)}}e^{-\gamma_j(x_1^{(j)}-\alpha)^2}\\
    &\quad \times e^{-i\pi (B^+Ax_1\cdot x_1+2C^T\xi_2\cdot x_1)},
\end{align*}
for almost every $(x_1,\xi_2)\in \ker(B)^\perp\times A(\ker(B))$. Writing $x_2=D^T\xi_2=\mathcal D\xi_2$ and setting $\Theta(x_2;\cdot):=\Theta_{\mathcal D^{-1}x_2}(\cdot)$, we obtain \eqref{flimitej}; moreover, $\Theta\in L^2(D^TA(\ker B)\times\bR^{r-1})$ because $f\in\lrd$ and the changes of variables used to prove \eqref{e3} have constant Jacobian.
It remains to treat the case in which $\ker B=\{0\}$. In this scenario we have that $A(\ker B)=\{0\}$ and $\ker B^{\perp}=R(B)=\rd$, so we only need one auxiliary function \begin{equation*}
    g(u):=f(Vu)e^{i\pi V^TB^{-1}AVu\cdot u},\quad u\in \rd.
\end{equation*}
Following the same steps as above, applying the classical directional Heisenberg uncertainty principle on $g$ one can easily derive the inequality
\begin{equation*}
			\left(\int_{\bR^d}|x^{(j)}-\alpha|^2 |f(x)|^2dx\right)^{1/2}\left(\int_{\bR^d}|\xi^{(j)}-\beta|^2|\widehat Sf(\xi)|^2d\xi\right)^{1/2}\geq\frac{\norm{f}_2^2}{4\pi}.
		\end{equation*}
Here $x^{(j)}$ is the $j$-th component of $x$ with respect to the orthonormal basis $\{v_1,\dots,v_d\}$ of $\rd=\ker B^{\perp}$ (e.g. the canonical basis), and $\xi^{(j)}$ is the $j$-th component of $\xi$ with respect to the basis $\{Bv_1,\dots,Bv_d\}$ of $R(B)=\rd$. The functions realizing equality are already described by \eqref{flimitej}, with $\Theta$ depending only on $(x^{(i)})_{i\neq j}$, and with $x_1=x$ and $x_2=0$. {Thus \eqref{e5}, with no additional geometric factor, covers both $\ker B\ne\{0\}$ and $\ker B=\{0\}$.}
\end{proof}

\begin{theorem}\label{Teo2}
   Under the assumptions of Theorem \ref{dirHeis}, with {$\alpha,\beta \in\bR$} replaced by
   $\alpha\in\ker(B)^\perp$, $\beta\in R(B)$, respectively,  we have
			\begin{equation}\label{e6}
			\left(\int_{\bR^d}|x_1-\alpha|^2 |f(x)|^2dx\right)^{1/2}\left(\int_{\bR^d}|\xi_1-\beta|^2|\widehat Sf(\xi)|^2d\xi\right)^{1/2}\geq{\frac{\mathrm{Tr}((B^TB)^{1/2})}{4\pi}\norm{f}_2^2},
		\end{equation}
		for every $f\in\lrd$. Furthermore, the constant in Theorem~\ref{Teo2} is sharp for every metaplectic operator
		$\widehat S$ with $\operatorname{rank}(B)=r\ge1$. If $\{v_1,\dots,v_r\}$ is chosen as an orthonormal eigenbasis of $B^TB|_{\ker B^{\perp}}$ and $\sigma_j(B)>0$ is the singular value associated with $v_j$, then, besides the zero function, equality holds precisely for the functions \begin{equation}\label{extremizers}
\begin{aligned}
f(x_1+x_2)=\Theta(x_2)\exp\Biggl(&
  2\pi i\sum_{j=1}^r\beta^{(j)}x_1^{(j)}
  -\sum_{j=1}^r\frac{\pi}{c\,\sigma_j(B)}
    (x_1^{(j)}-\alpha^{(j)})^2\\
  &-i\pi\bigl(B^+Ax_1\cdot x_1
    +2C^T\mathcal{D}^{-1}x_2\cdot x_1\bigr)\Biggr).
\end{aligned}
\end{equation}
For some $\Theta\in L^2(D^TA(\ker B))\setminus \{0\}$
\end{theorem}

\begin{proof}
    We choose the orthonormal basis of $\ker(B)^\perp$ given by the eigenvectors $\{v_1,\dots,v_r\}$ of $B^TB_{|\ker(B)^\perp}$ with corresponding eigenvalues $0<\lambda_1=\sigma_1(B)^2\leq\dots\leq\lambda_r=\sigma_r(B)^2$. In this way we have the basis $\{Bv_1,\dots,Bv_r\}$ on $R(B)$ which is orthogonal. In fact,
    \begin{equation*}
         Bv_i\cdot Bv_j=B^TBv_i\cdot v_j=\lambda_i v_i\cdot v_j=\lambda_i\delta_{ij}.
    \end{equation*}
   Furthermore,
   \begin{equation}\label{e7}
        |Bv_i|^2=Bv_i\cdot Bv_i=B^TBv_i\cdot v_i=\lambda_i.
    \end{equation}
   We rewrite the left-hand side of \eqref{e6} by expanding $(x_1-\alpha)$ in the basis $\{v_1,\dots,v_r\}$ and $(\xi_1-\beta)$ in the basis $\{Bv_1,\dots,Bv_r\}$. The orthogonality of the bases yields\begin{equation}
    \int_{\bR^d}|\displaystyle\sum_{i=1}^r(x_1^{(i)}-\alpha^{(i)})v_i|^2|f(x)|^2dx=\sum_{i=1}^r\int_{\bR^d}|(x_1^{(i)}-\alpha^{(i)})|^2 |f(x)|^2dx\label{e80}
    \end{equation}
and
\begin{equation}
       \int_{\bR^d}|\displaystyle\sum_{j=1}^r(\xi_1^{(j)}-\beta^{(j)})Bv_j|^2|\widehat Sf(\xi)|^2d\xi= \label{e8}\sum_{j=1}^r\sigma_j(B)^2\int_{\bR^d}|(\xi_1^{(j)}-\beta^{(j)})|^2|\widehat Sf(\xi)|^2d\xi.
    \end{equation}
   We define \begin{equation}\label{eq:h-k}
   	k_i:=\left(\int_{\bR^d}|(x_1^{(i)}-\alpha^{(i)})|^2 |f(x)|^2dx\right)^{1/2},\quad h_j:=\sigma_j(B)\left(\int_{\bR^d}|(\xi_1^{(j)}-\beta^{(j)})|^2|\widehat Sf(\xi)|^2d\xi\right)^{1/2},
   \end{equation}  
   {$k:=(k_i)_{i=1}^r$ and $h:=(h_i)_{i=1}^r\in\bR^r$}. Observe that \eqref{e80} and \eqref{e8} are $\|k\|_2^2$ and $\|h\|_2^2$, respectively. The Cauchy--Schwarz inequality {in $\bR^r$}, $ \|k\|_2 \|h\|_2\geq |k\cdot h|$,
    yields
   \begin{align*}
        &\left(\displaystyle\sum_{i=1}^r\int_{\bR^d}|(x_1^{(i)}-\alpha^{(i)})|^2|f(x)|^2{dx}\right)^{1/2}\left(\sum_{j=1}^r\sigma_j(B)^2\int_{\bR^d}|(\xi_1^{(j)}-\beta^{(j)})|^2|\widehat Sf(\xi)|^2d\xi\right)^{1/2}\\
        &\qquad\geq\quad\displaystyle\sum_{j=1}^r\sigma_j(B)\bigg(\int_{\bR^d}|(x_1^{(j)}-\alpha^{(j)})|^2 |f(x)|^2dx\bigg)^{1/2}\bigg(\int_{\bR^d}|\xi_1^{(j)}-\beta^{(j)}|^2|\widehat Sf(\xi)|^2d\xi\bigg)^{1/2}\\
        & \qquad\geq\quad{\frac{\displaystyle\sum_{j=1}^r \sigma_j(B)}{4\pi}\norm{f}_2^2}\\
        & \qquad =\quad {\frac{\mathrm{Tr}((B^TB)^{1/2})}{4\pi}\norm{f}_2^2},
    \end{align*}
    where the  last inequality follows from \eqref{e5}. By Theorem  \ref{dirHeis}, the last estimate is an equality if and only if the function $f$ takes the form \eqref{flimitej}, 
    for every $j=1,\dots,r$. Hence, for a.e. $x_2$, the function $x_1\mapsto f(x_1+x_2)e^{i\pi (B^+Ax_1\cdot x_1+2C^T\mathcal D^{-1}x_2\cdot x_1)}$ is, for every $j$, the product of a function of $(x_1^{(i)})_{i\ne j}$ and of the modulated Gaussian $e^{2\pi i\beta^{(j)}x_1^{(j)}}e^{-\gamma_j(x_1^{(j)}-\alpha^{(j)})^2}$ (here $\alpha^{(j)},\beta^{(j)}$ play the role of $\alpha,\beta$ in Theorem \ref{dirHeis}). Iterating over $j=1,\dots,r$, it is a constant (depending on $x_2$) times $\prod_{j=1}^r e^{2\pi i\beta^{(j)}x_1^{(j)}}e^{-\gamma_j(x_1^{(j)}-\alpha^{(j)})^2}$.
    So, up to renaming of $\Theta (x_2)$, where $\Theta \in L^2(D^TA(\ker B))$, we obtain the expression:\begin{equation*}
        {f(x_1+x_2)=\Theta (x_2)\exp\!\left(2\pi i \displaystyle\sum_{j=1}^r\beta^{(j)} x_1^{(j)}-\displaystyle\sum_{j=1}^r\gamma_j(x_1^{(j)}-\alpha^{(j)})^2-i\pi (B^+Ax_1\cdot x_1+2C^T\mathcal D^{-1}x_2\cdot x_1)\right)},
    \end{equation*}
    or, equivalently,
    \begin{equation}
    	{f(x_1+D^T\xi_2)
    	= \zeta(\xi_2)\exp\!\left(2\pi i \displaystyle\sum_{j=1}^r\beta^{(j)} x_1^{(j)}-\displaystyle\sum_{j=1}^r\gamma_j(x_1^{(j)}-\alpha^{(j)})^2-i\pi (B^+Ax_1\cdot x_1+2C^T\xi_2\cdot x_1)\right)},
    \end{equation}
 for some $\gamma_j>0$ (possibly depending on $j$), and $\zeta \in L^2(A(\ker B))$.   
   
The inequality\begin{align*}
\|k\|_2 \|h\|_2\geq | k\cdot h|
\end{align*}
becomes an equality if and only if {there exists $c>0$ such that $k=ch$}. This gives \begin{align*}
    \int_{\bR^d}|(x_1^{(j)}-\alpha^{(j)})|^2 |f(x)|^2dx=c^2\sigma_j(B)^2\left(\int_{\bR^d}|(\xi_1^{(j)}-\beta^{(j)})|^2|\widehat Sf(\xi)|^2d\xi\right),\quad \forall j=1,\dots,r.
\end{align*}
The equivalent condition for $g_{\xi_2}$ defined in \eqref{gu} becomes
{
\begin{equation}\label{cond1}
	\int_{A(\ker(B))}\int_{\bR^r}|u_j-\alpha^{(j)}|^2 | g_{\xi_2}(u)|^2dud\xi_2
    =c^2\sigma_j(B)^2 \int_{A(\ker(B))}\int_{\bR^r}|\eta_j-\beta^{(j)}|^2|\widehat  g_{\xi_2}(\eta)|^2d\eta d\xi_2.
\end{equation}
}
Now, the function takes the form
\[
g_{\xi_2}(u)
= \zeta(\xi_2)\,L\,
e^{2\pi i \sum_{j=1}^r\beta^{(j)} u_j}\,
\exp\!\Big(-\sum_{j=1}^r \gamma_j (u_j-\alpha^{(j)})^2\Big),
\qquad u\in\mathbb R^r,
\]
for a suitable constant $L\in\bC$.

Since $|\zeta (\xi_2)|^2$ factors out from both sides of \eqref{cond1}, condition \eqref{cond1}
reduces to an identity in $\R^r$:
\[
\int_{\R^r} |u_j-\alpha^{(j)}|^2 \, |G(u)|^2\,du
=
{c^2 \sigma_j(B)^2}
\int_{\R^r} |\eta_j-\beta^{(j)}|^2 \, |\widehat G(\eta)|^2\,d\eta,
\qquad j=1,\dots,r,
\]
where
\[
G(u)=e^{2\pi i \sum_{j=1}^r\beta^{(j)} u_j}\exp\!\Big(-\sum_{j=1}^r \gamma_j (u_j-\alpha^{(j)})^2\Big).
\]

By separability of $G(u)$ and the one-dimensional Gaussian computations, we have for each $j$:
\[
\int_{\R^r} |u_j-\alpha^{(j)}|^2 |G(u)|^2\,du
=\frac{1}{4\gamma_j}\,\|G\|_2^2,
\qquad
\int_{\R^r} |\eta_j-\beta^{(j)}|^2 |\widehat G(\eta)|^2\,d\eta
=\frac{\gamma_j}{4\pi^2}\,\|G\|_2^2.
\]
Hence \eqref{cond1} holds if and only if
\[
\frac{1}{4\gamma_j}\|G\|_2^2
=
{c^2 \sigma_j(B)^2} \frac{\gamma_j}{4\pi^2}\|G\|_2^2
\quad\Longleftrightarrow\quad
\gamma_j
=
{\frac{\pi}{c\,\sigma_j(B)}},
\qquad j=1,\dots,r.
\]
Therefore, choosing $\gamma_j$ as above makes the Cauchy--Schwarz step in the proof  an equality, and the constant in \eqref{e6} is sharp.
\end{proof}
\begin{remark}[Sharpness and structure of extremizers]
	Observe that the extremizers of Theorem~\ref{Teo2} are in general
	\emph{anisotropic} Gaussians, whose anisotropy reflects the singular spectrum
	of the block $B$ of the underlying symplectic matrix.
	The isotropic Gaussian case occurs only when $\sigma_1(B)=\cdots=\sigma_r(B)$,
	which includes, as a special case, the Fourier transform, detailed below.
\end{remark}

\begin{example}[Fourier transform]
	Let
	\[
	J=
	\begin{pmatrix}
		0 & I \\
		- I & 0
	\end{pmatrix},
	\]
    that is the case where $A=0$, $B=I$, $C=-I$ and $D=0$. Then, $r=\operatorname{rank}(B)=d$ and $\sigma_j(B)=1$ for every $j$.
	The metaplectic operator associated with $J$ is (up to a phase) the Fourier transform.
	
	The inequality \eqref{e6} becomes
	\[
	\Big(\int_{\mathbb R^d} |x-\alpha|^2 |f(x)|^2\,dx\Big)^{1/2}
	\Big(\int_{\mathbb R^d} |\xi-\beta|^2 |\widehat f(\xi)|^2\,d\xi\Big)^{1/2}
	\ge
	\frac{d}{4\pi}\,\|f\|_2^2,
	\]
	which coincides with the classical Heisenberg--Pauli--Weyl inequality. Equality holds for the Gaussian
	\[
	f(x)=e^{2\pi i\,\beta\cdot x}\,e^{-\pi\gamma |x-\alpha|^2},
	\qquad \gamma>0.
	\]
	In fact, one computes
	\[
	\int |x-\alpha|^2 |f(x)|^2\,dx
	= \frac{d}{4\pi\gamma}\,\|f\|_2^2,
	\qquad
	\int |\xi-\beta|^2 |\widehat f(\xi)|^2\,d\xi
	= \frac{d\gamma}{4\pi}\,\|f\|_2^2,
	\]
	so that
	\[
	\|\,|x-\alpha|f\|_2\,\|\,|\xi-\beta|\widehat f\|_2
	=
	\frac{d}{4\pi}\,\|f\|_2^2.
	\]
	Hence, the constant in Theorem~\ref{Teo2} is sharp and recovers the classical case.
\end{example}

\begin{example}\label{ex4.8}
    The Fourier multiplier
    \begin{equation}
        \mathfrak{m}_Pf(x)=\cF^{-1}(e^{-i\pi Py\cdot y}\widehat f)(x), \qquad P\in\bR^{d\times d},\, P=P^\top
    \end{equation}
    is a metaplectic operator and its projection is
    \begin{equation}
        V_{P}^T=\begin{pmatrix}
            I_d & P\\
            0_d & I_d
        \end{pmatrix}.
    \end{equation}
    {In this case, equation \eqref{e6} reads as}
    \begin{equation}
        \left(\int_{\bR^d}|x_1-\alpha|^2 |f(x)|^2dx\right)^{1/2}\left(\int_{\bR^d}|\xi_1-\beta|^2|\mathfrak{m}_Pf(\xi)|^2d\xi\right)^{1/2}\geq\frac{\mathrm{Tr}((P^TP)^{1/2})}{4\pi}\norm{f}_2^2.
    \end{equation}
    As a concrete example, let us apply our result to the Schr\"odinger equation of the free particle. Consider the Cauchy problem
    \begin{equation}\label{Schro1}
        \begin{cases}
        i\partial_tu(t,x)=\Delta_xu(t,x) & \text{$t\in\bR$, $x\in\rd$},\\
        u(0,x)=u_0(x),
        \end{cases}
    \end{equation}
    with $u_0\in\cS(\rd)$. The propagator $e^{-it\Delta_x}$ of \eqref{Schro1} is the one-parameter subgroup of Fourier multipliers
    \begin{equation}
        e^{-it\Delta_x}u_0=\mathfrak{m}_{-4\pi t I_d}u_0=\cF^{-1}(e^{i\pi (4\pi t)|\cdot|^2}\widehat u_0),
    \end{equation}
    having projection given for every $t\in\bR$ by
    \begin{equation}
        V^\top_{-4\pi t I_d}=\begin{pmatrix}
            I_d & -4\pi t I_d\\
            0_d & I_d
        \end{pmatrix}.
    \end{equation}
    Since $\sigma_j(B)=4\pi|t|$ for every $j=1,\ldots,d$, Heisenberg's inequality reads in this case as
    \begin{equation}
        \left(\int_{\bR^d}|x-\alpha|^2 |u_0(x)|^2dx\right)^{1/2}\left(\int_{\bR^d}|\xi-\beta|^2|e^{-it\Delta_x}u_0(\xi)|^2d\xi\right)^{1/2}\geq d|t|\norm{u_0}_2^2,
    \end{equation}
    this inequality encompasses both the case where the uncertainty principle holds, i.e., $t\neq0$, and the case where it fails, $t=0$.
\end{example}

\begin{example}
    Let $\cJ\subseteq\{1,\ldots,d\}$ be a non-empty subset of indices. The metaplectic operator, defined up to a phase factor, having projection
    \begin{equation}
        \Pi_\cJ=\begin{pmatrix}
            I_d-I_{\cJ} & I_\cJ\\
            -I_\cJ & I_d-I_\cJ
        \end{pmatrix}, 
    \end{equation}
    where $I_\cJ$ is the diagonal matrix with $j$-th diagonal index $=1$ if $j\in\cJ$ and $=0$ otherwise, is the so-called partial Fourier transform with respect to the variables indexed by $\cJ$, expressed as
    \begin{equation}
        \cF_{\cJ}f(\xi_1+\xi_2)=\int_{R(I_\cJ)} f(x_1+\xi_2)e^{-2\pi i x_1\cdot\xi_1}dx_1, \qquad f\in\cS(\rd),
    \end{equation}
    where $r=\mathrm{card}(\cJ)$, while $x_1=I_{\cJ}x$, $x_2=x-x_1$ and similarly for $\xi_1,\xi_2$. Heisenberg's inequality in this case reads as
    \begin{equation}
        \Big(\int_{\mathbb R^d} |x_1-\alpha|^2 |f(x)|^2\,dx\Big)^{1/2}
	\Big(\int_{\mathbb R^d} |\xi_1-\beta|^2 |\cF_\cJ f(\xi)|^2\,d\xi\Big)^{1/2}
	\ge
	\frac{r}{4\pi}\,\|f\|_2^2.
    \end{equation}
\end{example}

\begin{example}\label{ex4.10}
    Consider the Quantum Harmonic Oscillator
    \begin{equation}
        \begin{cases}
            i\partial_tu(t,x)=\widehat Hu(t,x), & \text{if $t\in\bR$, $x\in\rd$},\\
            u(0,x)=u_0(x),
        \end{cases}
    \end{equation}
    for $u_0\in\cS(\rd)$ and 
    \begin{equation}
        \widehat{H}=-\dfrac{1}{4\pi}\Delta_x+\pi\omega^2|x|^2, \qquad \omega>0.
    \end{equation} 
    The propagator $e^{-it\widehat H}$ is the one-parameter subgroup of $\Mp(d,\bR)$ with projections $S_t=\pi^{Mp}(e^{-it\widehat H})$ given by
    \begin{equation}\label{defStHO}
        S_t=\begin{pmatrix}
            \cos(\omega t)I_d & \frac 1 \omega \sin(\omega t)I_d\\
            -\omega\sin(\omega t) I_d & \cos(\omega t)I_d
        \end{pmatrix},
    \end{equation}
    that is, $e^{-it\widehat H}$ is, up to the dilation $\mathfrak T_{\sqrt\omega I_d}$, the fractional Fourier transform of angle $\omega t$, explicitly given up to a phase by
    \begin{equation}
        e^{-it\widehat H}u_0(x)=
        \begin{cases}
            \omega^{d/2}(1-i\cot(\omega t))^{d/2}\int_{\rd}u_0(y)e^{i\pi\omega\cot(\omega t)(|x|^2+|y|^2)}e^{-2\pi i\omega x\cdot y/\sin(\omega t)}dy & \text{if $\sin(\omega t)\neq0$},\\
            u_0((-1)^kx) & \text{if $t=\frac{k\pi}{\omega}$, $k\in\bZ$}.
        \end{cases}
    \end{equation}
    Clearly, the upper-right block in \eqref{defStHO} does not vanish only in the first scenario, where Heisenberg's inequality reads as
    \begin{equation}\label{UPHerm}
        \Big(\int_{\mathbb R^d} |x-\alpha|^2 |u_0(x)|^2\,dx\Big)^{1/2}
	\Big(\int_{\mathbb R^d} |\xi-\beta|^2 |e^{-it\widehat H} u_0(\xi)|^2\,d\xi\Big)^{1/2}
	\ge
	\frac{d|\sin(\omega t)|}{4\omega\pi}\,\|u_0\|_2^2,
    \end{equation}
    since the singular values of the upper-right block are all equal to $|\sin(\omega t)|/\omega$. Observe that \eqref{UPHerm} also includes the case where $\sin(\omega t)=0$ and the uncertainty principle fails.
\end{example}

We conclude this section by commenting on a remarkable directional version of Heisenberg's UP proven by Dias, de Gosson and Prata in \cite{CGP2024}. 
Specifically, we refer to Theorem 3 therein, that we rephrase as follows.% CHECK theorem number in [CGP2024]

\begin{theorem}
    Let $\widehat S\in\Mp(d,\bR)$ have projection $S\in\Sp(d,\bR)$ with blocks \eqref{blockS}. Then, for every $f\in\lrd$, every $j,k=1,\ldots,d$ and every $\alpha,\beta\in\bR$, 
    \begin{equation}\label{GosH}
        \Big(\int_{\rd}(x_j-\alpha)^2|f(x)|^2dx\Big)^{1/2}\Big(\int_{\rd}(\xi_k-\beta)^2|\widehat Sf(\xi)|^2d\xi\Big)^{1/2}\geq\frac{|B_{kj}|}{4\pi}\norm{f}_2^2,
    \end{equation}
    where $B_{kj}$ is the entry of $B$ in the $k$-th row and $j$-th column.
\end{theorem}
A comparison between \eqref{e5} and \eqref{GosH} is therefore in order. While the two inequalities are formally very similar, they differ in how the directional information associated with the metaplectic operator is distributed.
Inequality \eqref{e5} encodes this directional information in its left-hand side. This is a consequence of our choice of coordinates, which is consistent with the phase-space geometry induced by the metaplectic transformation and by its effective directions. Accordingly, the constant on the right-hand side does not depend on the index $j$.
In \eqref{GosH}, on the contrary, the relevant directional information is entirely contained in the right-hand side, through the factor $|B_{k j}|$, since the coordinates in the left-hand side are the Cartesian coordinates, and they are not chosen according to $\widehat S$.
Besides being more general, as the indices $j$ and $k$ need not coincide, this formulation reveals that the entries of $B$ measure the correlation between the Cartesian variables of $f$ and those of $\widehat S f$. 

\section{Beurling and Morgan-type UPs}
In this Section we generalize Beurling's UP of Theorem \ref{BeurlingFT} and Morgan's one of Theorem \ref{classicMorgan}. We start with the Beurling-type result.
\begin{theorem}\label{Thrm4.1}
	Let $\widehat S$ be a metaplectic operator on $L^2(\rd)$ with projection \eqref{blockS} such that  $1\leq r= \mathrm{rank}(B)\leq d$, let $N\ge0$ and $f\in\lrd$. Consider the decompositions \eqref{D1} and \eqref{D2}. Then
		\begin{equation}\label{beurlingmetap}
		\int_{\ker(B)^\perp}\int_{\ker(B)^\perp}
		\frac{|f(x_1+D^T\xi_2)||\widehat Sf(B\zeta_1+\xi_2)|}{(1+|x_1|+|\zeta_1|)^N}e^{2\pi|x_1\cdot\zeta_1|}dx_1d\zeta_1<\infty
	\end{equation}
for a.e. $\xi_2\in A(\ker(B))$
if and only if, for a.e. $x_2\in D^TA(\ker B)$,
$$f(x_1+x_2)=\widetilde q_{x_2}(x_1)e^{-\pi\widetilde M_{x_2} x_1\cdot x_1}e^{-i\pi (B^+Ax_1\cdot x_1+2C^T\mathcal D^{-1}x_2\cdot x_1)},\qquad \text{a.e. }x_1\in\ker(B)^\perp,$$
where  $\widetilde q_{x_2}$ is a polynomial of degree $<(N-r)/2$ on $\ker(B)^\perp$ and $\widetilde M_{x_2}$ is a real symmetric matrix, positive definite on $\ker(B)^\perp$ (both possibly depending on $x_2$). In particular, $f=0$ if $N\le r$.
\end{theorem}
\begin{proof}
	Fix $\xi_2\in A(\ker(B))$. Let $ g_{\xi_2}$ be defined as in \eqref{gu}, so that $|\widehat Sf|$ and $|\widehat  g_{\xi_2}|$ are related according to \eqref{relSfhatg}. The Beurling condition \eqref{beurlingcondition} for $ g_{\xi_2}$ gives
	\begin{align*}
I:=&\int_{\bR^{2r}}\frac{| g_{\xi_2}(u)\widehat  g_{\xi_2}(\eta)|}{(1+|u|+|\eta|)^N}e^{2\pi|u\cdot\eta|}dud\eta=\frac{1}{\mu_S}\int_{\bR^{2r}}\frac{|f(Vu+D^T\xi_2)||\widehat Sf(BV\eta+\xi_2)|}{(1+|u|+|\eta|)^N}e^{2\pi|u\cdot\eta|}dud\eta\\
		&=\frac{1}{\sigma(B)\mu_S}\int_{\ker(B)^\perp}\int_{R(B)}\frac{|f(x_1+D^T\xi_2)||\widehat Sf(\xi_1+\xi_2)|}{(1+|V^Tx_1|+|V^TB^+\xi_1|)^N}e^{2\pi|V^Tx_1\cdot V^TB^+\xi_1|}dx_1d\xi_1\\
		&=\frac{1}{\sigma(B)\mu_S}\int_{\ker(B)^\perp}\int_{R(B)}\frac{|f(x_1+D^T\xi_2)||\widehat Sf(\xi_1+\xi_2)|}{(1+|x_1|+|B^+\xi_1|)^N}e^{2\pi|x_1\cdot B^+\xi_1|}dx_1d\xi_1\\
		&=\frac{1}{\mu_S}\int_{\ker(B)^\perp}\int_{\ker(B)^\perp}\frac{|f(x_1+D^T\xi_2)||\widehat Sf(B\xi'_1+\xi_2)|}{(1+|x_1|+|\xi'_1|)^N}e^{2\pi|x_1\cdot\xi_1'|}dx_1d\xi'_1\\
	\end{align*}
where we set $\xi'_1= B^+\xi_1\in \ker(B)^\perp$, $d\xi_1=\sigma(B)d\xi'_1$ and we used that $VV^T$ is the identity on $\ker(B)^\perp$. Thus $I$ is, up to the factor $1/\mu_S$, the integral in \eqref{beurlingmetap} (with $\zeta_1=\xi_1'$).
By Theorem \ref{BeurlingFT} in dimension $r$, for a.e. $\xi_2$ (namely, those for which $g_{\xi_2}\in L^2(\bR^r)$ and \eqref{relSfhatg} holds), $I<\infty$ if and only if
$ g_{\xi_2}(u)=q_{\xi_2}(u)e^{-\pi M_{\xi_2}u\cdot u}$, with $q_{\xi_2}$ a polynomial of degree $<(N-r)/2$ and $M_{\xi_2}\in\bR^{r\times r}$ symmetric positive definite, that is
$$q_{\xi_2}(u)e^{-\pi M_{\xi_2}u\cdot u}=f(Vu+D^T\xi_2)e^{i\pi (V^TB^+AVu\cdot u+2V^TC^T\xi_2\cdot u)}, \qquad u\in\bR^r.$$
Setting $u=V^Tx_1$, $x_1\in\ker(B)^\perp$, and using $VV^Tx_1=x_1$, this is equivalent to
$$f(x_1+D^T\xi_2)=q_{\xi_2}(V^Tx_1)e^{-\pi VM_{\xi_2}V^Tx_1\cdot x_1}e^{-i\pi (B^+Ax_1\cdot x_1+2C^T\xi_2\cdot x_1)}.$$
We set $x_2=D^T\xi_2=\mathcal D\xi_2\in D^TA(\ker(B))$, $\widetilde q_{x_2}:=q_{\mathcal D^{-1}x_2}\circ V^T$ and $\widetilde M_{x_2}:=VM_{\mathcal D^{-1}x_2}V^T$, and we obtain the thesis.
\end{proof}

%\section{Morgan type uncertainty principle}

Next, we extend Morgan's UP to metaplectic operators. 
 Consider a metaplectic operator $\widehat S$ on $L^2(\rd)$ with projection \eqref{blockS}, such that  $1\leq r= \mathrm{rank}(B)\leq d$. We fix an arbitrary orthonormal basis $\{v_1,\dots, v_r\}$ of $\ker B^{\perp}$ and we denote by $w^{(j)}$ the $j$-th component of an arbitrary vector $w\in\ker B^{\perp}$ with respect to that basis, while, if $y\in\bR^r$, we denote by $y_j$, the $j$-th component of $y$ with respect to the canonical basis of $\bR^r$. Moreover we consider again the two decompositions of $\rd$ given by \eqref{D1} and \eqref{D2}.
      
      \begin{theorem}\label{5.1} Let $1<p<2$ and let $q$ be the conjugate exponent. Assume that $f\in L^2(\rd)$ satisfies the two conditions:\begin{align}\label{Morgan1}
        &\int_{\ker B^{\perp}} |f(x_1+x_2)|e^{\frac{2\pi}{p}a^p|x_1^{(j)}|^p}dx_1<\i\quad \text{for a.e. }x_2\in D^T(A(\ker B)),\\ &\int_{\ker B^{\perp}} |\widehat S f(B\zeta_1+\xi_2)|e^{\frac{2\pi}{q}b^q|\zeta_1^{(j)}|^q}d\zeta_1<\i\quad \text{for a.e. }\xi_2\in A(\ker B),\label{Morgan2}
    \end{align}
    for some $j=1,\dots,r$ and for some $a,b>0$. Then,\begin{itemize}
        \item [i)]$ab>|\cos(p\pi/2)|^{1/p}\implies f=0.$
        \item [ii)]$ab<|\cos(p\pi/2)|^{1/p}\implies \exists $ a dense subset $\mathfrak D\subseteq \lrd$ of functions satisfying the above conditions for all $j=1,\dots,r$.
    \end{itemize}
\end{theorem}

\begin{proof}
   In what follows, let \(V:\mathbb R^r\to\ker B^\perp\) be the isometric parametrization determined by \(Ve_j=v_j\), \(j=1,\ldots,r\), and $g_{\xi_2}$ be the  auxiliary function   defined in \eqref{gu}.
   \par
     Condition \eqref{Morgan1} can be rephrased as follows:  for almost every $x_2=D^T\xi_2\in D^T(A(\ker B))$,\begin{align*}
        \int_{\ker B^{\perp}} |f(x_1+x_2)|e^{\frac{2\pi}{p}a^p|x_1^{(j)}|^p}dx_1=&\int_{\bR^r} |f(Vu+x_2)|e^{\frac{2\pi}{p}a^p|u_j|^p}du\\
        =&\int_{\bR^r} |f(Vu+D^T\xi_2)|e^{\frac{2\pi}{p}a^p|u_j|^p}du\\
        =&\int_{\bR^r} |g_{\xi_2}(u)|e^{\frac{2\pi}{p}a^p|u_j|^p}du<\i.
    \end{align*}
    Similarly, Condition \eqref{Morgan2} becomes: for almost every $\xi_2\in A(\ker B)$,\begin{align*}
        \int_{\ker B^{\perp}} |\widehat S f(B\zeta_1+\xi_2)|e^{\frac{2\pi}{q}b^q|\zeta_1^{(j)}|^q}d\zeta_1=& \int_{\bR^r} |\widehat S f(BV\eta+\xi_2)|e^{\frac{2\pi}{q}b^q|\eta_j|^q}d\eta\\
        =& \mu_S\int_{\bR^r} |\widehat g_{\xi_2}(\eta)|e^{\frac{2\pi}{q}b^q|\eta_j|^q}d\eta<\i.
    \end{align*}
    Hence, for some $j=1,\dots,r$, and  $a,b>0$, $g_{\xi_2}\in L^2(\bR^r)$ satisfies:\begin{equation}\label{morgancond}
        \int_{\bR^r} |g_{\xi_2}(u)|e^{\frac{2\pi}{p}a^p|u_j|^p}du<\i\quad \text{and}\quad \int_{\bR^r} |\widehat g_{\xi_2}(\eta)|e^{\frac{2\pi}{q}b^q|\eta_j|^q}d\eta<\i,
    \end{equation}
     for $a.e.$ $\xi_2\in A(\ker B)$. Therefore, by the Gel'fand-Shilov type UP in Theorem \ref{classicMorgan} we infer the following conditions:\\
\text{i)} if $ab>|\cos(p\pi/2)|^{1/p}$ then $g_{\xi_2}=0$ for $a.e.\ \xi_2\in A(\ker B)$, which, together with the decomposition $x=x_1+x_2$ in \eqref{D1}  implies  $f=0$.\par\noindent
       $\text{ii)}$ if $ab<|\cos(p\pi/2)|^{1/p}$ there exists  a dense subset $H$ of $L^2(\bR^r)$ such that every $g\in H$ satisfies the conditions in \eqref{morgancond}, for all $j=1,\dots,r$. Now, define
        \begin{equation*}
            \mathfrak{D}:=\text{span}\{ f\in\lrd\,:\,f(Vu+D^T\xi_2)=h(\xi_2)g(u)e^{-i\pi(V^TB^+AVu\cdot u+2V^TC^T\xi_2\cdot u)}\}
        \end{equation*}
        where $  g\in H$, $h\in \cC_c(A(\ker B))$.
       For a function $f$ as above we have:\begin{align*}
            \int_{\ker B^{\perp}} |f(x_1+x_2)|e^{\frac{2\pi}{p}a^p|x_1^{(j)}|^p}dx_1=&\int_{\bR^r} |f(Vu+D^T\xi_2)|e^{\frac{2\pi}{p}a^p|u_j|^p}du\\
            =&\int_{\bR^r} |h(\xi_2)||g(u)|e^{\frac{2\pi}{p}a^p|u_j|^p}du\\
            \leq & \norm{h}_{L^{\i}(A(\ker B))}\int_{\bR^r}|g(u)|e^{\frac{2\pi}{p}a^p|u_j|^p}du<\i,
        \end{align*}
      Using the integral representation of the metaplectic operator $\widehat{S}$ in \eqref{f2} we can write \begin{equation*}
            |\widehat Sf(BV\eta+\xi_2)|=\mu_S |h(\xi_2)||\widehat g(\eta)|,\quad \text{for a.e. } \eta\in\bR^r,\ \xi_2\in A(\ker B).
        \end{equation*}
        Therefore, proceeding as before, \begin{equation*}
            \int_{\ker B^{\perp}} |\widehat S f(B\zeta_1+\xi_2)|e^{\frac{2\pi}{q}b^q|\zeta_1^{(j)}|^q}d\zeta_1<\i\quad \text{for a.e. }\xi_2\in A(\ker B).
        \end{equation*}
    
By the triangle inequality and the linearity of $\widehat S$, every linear combination of such functions still satisfies the two Morgan's conditions \eqref{Morgan1} and \eqref{Morgan2}. Hence, every function in $\cD$ satisfies the two Morgan conditions.
Finally, consider the operator $U$ defined on $L^2(\bR^r\times A(\ker B))$ by
\begin{equation*}
UF(Vu+D^T\xi_2):=F(u,\xi_2)e^{-i\pi(V^TB^+AVu\cdot u+2V^TC^T\xi_2\cdot u)},\qquad u\in\bR^r,\ \xi_2\in A(\ker B).
\end{equation*}
By the changes of variables used in the proof of \eqref{e3}, $\norm{UF}_{L^2(\rd)}^2=q_{\ker B}(A)^{-1}\norm{F}^2_{L^2(\bR^r\times A(\ker B))}$, so that $q_{\ker B}(A)^{1/2}U$ is unitary onto $L^2(\rd)$ and maps dense subspaces onto dense subspaces. Since $\mathfrak D=U\big(\mathrm{span}\{g\otimes h:\ g\in H,\ h\in \cC_c(A(\ker B))\}\big)$ and this span is dense in $L^2(\bR^r\times A(\ker B))$ by Corollary \ref{densita} (if $r=d$, then $A(\ker B)=\{0\}$ and $\mathfrak D=U(\mathrm{span}\, H)$, with no need of Corollary \ref{densita}), the density of $\mathfrak{D}$ in $L^2(\rd)$ follows.  
\end{proof}

\section*{Acknowledgements}
The  authors have been supported by the Gruppo Nazionale per l’Analisi Matematica, la Probabilità e le loro Applicazioni (GNAMPA) of the Istituto Nazionale di Alta Matematica (INdAM). Gianluca Giacchi has been supported by the SNSF starting grant ``Multiresolution methods for unstructured data” (TMSGI2 211684).

\bibliographystyle{abbrv}
\bibliography{bibliography}

@article {BDsurvey,
    AUTHOR = {Bonami, Aline and Demange, Bruno},
     TITLE = {A survey on uncertainty principles related to quadratic forms},
   JOURNAL = {Collect. Math.},
  FJOURNAL = {Universitat de Barcelona. Collectanea Mathematica},
      YEAR = {2006},
    NUMBER = {Vol. Extra},
     PAGES = {1--36},
      ISSN = {0010-0757,2038-4815},
   MRCLASS = {42B10 (43A25 46F12)},
  MRNUMBER = {2264204},
MRREVIEWER = {Slaim\ Ben Farah},
}

@article {BDJ,
    AUTHOR = {Bonami, Aline and Demange, Bruno and Jaming, Philippe},
     TITLE = {Hermite functions and uncertainty principles for the {F}ourier
              and the windowed {F}ourier transforms},
   JOURNAL = {Rev. Mat. Iberoamericana},
  FJOURNAL = {Revista Matem\'{a}tica Iberoamericana},
    VOLUME = {19},
      YEAR = {2003},
    NUMBER = {1},
     PAGES = {23--55},
      ISSN = {0213-2230},
   MRCLASS = {42B10 (33C45 33C47 94A12)},
  MRNUMBER = {1993414},
MRREVIEWER = {H.\ S. P. Shrivastava},
       DOI = {10.4171/RMI/337},
       URL = {https://doi.org/10.4171/RMI/337},
}

@article {CF2015,
    AUTHOR = {Cassano, B. and Fanelli, L.},
     TITLE = {Sharp {H}ardy uncertainty principle and {G}aussian profiles of
              covariant {S}chr\"{o}dinger evolutions},
   JOURNAL = {Trans. Amer. Math. Soc.},
  FJOURNAL = {Transactions of the American Mathematical Society},
    VOLUME = {367},
      YEAR = {2015},
    NUMBER = {3},
     PAGES = {2213--2233},
      ISSN = {0002-9947,1088-6850},
   MRCLASS = {35Q41 (35A01 35J10)},
  MRNUMBER = {3286512},
MRREVIEWER = {William\ R.\ Green},
       DOI = {10.1090/S0002-9947-2014-06383-6},
       URL = {https://doi.org/10.1090/S0002-9947-2014-06383-6},
}

@article {CF2017,
    AUTHOR = {Cassano, B. and Fanelli, L.},
     TITLE = {Gaussian decay of harmonic oscillators and related models},
   JOURNAL = {J. Math. Anal. Appl.},
  FJOURNAL = {Journal of Mathematical Analysis and Applications},
    VOLUME = {456},
      YEAR = {2017},
    NUMBER = {1},
     PAGES = {214--228},
      ISSN = {0022-247X,1096-0813},
   MRCLASS = {35J10 (35B60)},
  MRNUMBER = {3680964},
       DOI = {10.1016/j.jmaa.2017.06.067},
       URL = {https://doi.org/10.1016/j.jmaa.2017.06.067},
}

@article {HUP,
    AUTHOR = {Cordero, Elena and Giacchi, Gianluca and Malinnikova, Eugenia},
     TITLE = {Hardy's uncertainty principle for {S}chr\"{o}dinger equations
              with quadratic {H}amiltonians},
   JOURNAL = {J. Lond. Math. Soc. (2)},
  FJOURNAL = {Journal of the London Mathematical Society. Second Series},
    VOLUME = {111},
      YEAR = {2025},
    NUMBER = {4},
     PAGES = {Paper No. e70134, 32},
      ISSN = {0024-6107,1469-7750},
   MRCLASS = {42A38 (35B05 35S30)},
  MRNUMBER = {4886668},
       DOI = {10.1112/jlms.70134},
       URL = {https://doi.org/10.1112/jlms.70134},
}

@article {CEKPV2010,
    AUTHOR = {Cowling, M. and Escauriaza, L. and Kenig, C. E. and Ponce, G.
              and Vega, L.},
     TITLE = {The {H}ardy uncertainty principle revisited},
   JOURNAL = {Indiana Univ. Math. J.},
  FJOURNAL = {Indiana University Mathematics Journal},
    VOLUME = {59},
      YEAR = {2010},
    NUMBER = {6},
     PAGES = {2007--2025},
      ISSN = {0022-2518,1943-5258},
   MRCLASS = {42B30},
  MRNUMBER = {2919746},
MRREVIEWER = {E.\ K.\ Narayanan},
       DOI = {10.1512/iumj.2010.59.4395},
       URL = {https://doi.org/10.1512/iumj.2010.59.4395},
}

@article {CGP2024,
    AUTHOR = {Dias, N. C. and de Gosson, M. A. and Prata, J. N.},
     TITLE = {A metaplectic perspective of uncertainty principles in the
              linear canonical transform domain},
   JOURNAL = {J. Funct. Anal.},
  FJOURNAL = {Journal of Functional Analysis},
    VOLUME = {287},
      YEAR = {2024},
    NUMBER = {4},
     PAGES = {Paper No. 110494, 54},
      ISSN = {0022-1236,1096-0783},
   MRCLASS = {42B10 (81S07)},
  MRNUMBER = {4748757},
MRREVIEWER = {Luigi\ Rodino},
       DOI = {10.1016/j.jfa.2024.110494},
       URL = {https://doi.org/10.1016/j.jfa.2024.110494},
}

@book {deGosSGQM,
    AUTHOR = {de Gosson, Maurice},
     TITLE = {Symplectic Geometry and Quantum Mechanics},
    SERIES = {Operator Theory: Advances and Applications},
    VOLUME = {166},
      NOTE = {Advances in Partial Differential Equations (Basel)},
 PUBLISHER = {Birkh\"{a}user Verlag, Basel},
      YEAR = {2006},
     PAGES = {xx+367},
      ISBN = {978-3-7643-7574-4; 3-7643-7574-4},
   MRCLASS = {81S10 (53D05 53D12 81Q20)},
  MRNUMBER = {2241188},
MRREVIEWER = {Brian\ C.\ Hall},
       DOI = {10.1007/3-7643-7575-2},
       URL = {https://doi.org/10.1007/3-7643-7575-2},
}

@article {EKPV2006,
    AUTHOR = {Escauriaza, Luis and Kenig, Carlos E. and Ponce, Gustavo and
              Vega, Luis},
     TITLE = {The sharp {H}ardy uncertainty principle for {S}chr\"{o}dinger
              evolutions},
   JOURNAL = {Duke Math. J.},
  FJOURNAL = {Duke Mathematical Journal},
    VOLUME = {155},
      YEAR = {2010},
    NUMBER = {1},
     PAGES = {163--187},
      ISSN = {0012-7094,1547-7398},
   MRCLASS = {35Q41 (35B45)},
  MRNUMBER = {2730375},
MRREVIEWER = {Boris\ Ettinger},
       DOI = {10.1215/00127094-2010-053},
       URL = {https://doi.org/10.1215/00127094-2010-053},
}

@article {BVJLMS2026,
    AUTHOR = {Fern\'{a}ndez-Bertolin, Aingeru and Vega, Luis},
     TITLE = {A theorem concerning {F}ourier transforms: {A} survey},
   JOURNAL = {J. Lond. Math. Soc. (2)},
  FJOURNAL = {Journal of the London Mathematical Society. Second Series},
    VOLUME = {113},
      YEAR = {2026},
    NUMBER = {1},
     PAGES = {Paper No. e70390},
      ISSN = {0024-6107,1469-7750},
   MRCLASS = {Prelim},
  MRNUMBER = {5012453},
       DOI = {10.1112/jlms.70390},
       URL = {https://doi.org/10.1112/jlms.70390},
}

@article {FMBAMS21,
    AUTHOR = {Fern\'{a}ndez-Bertolin, Aingeru and Malinnikova, Eugenia},
     TITLE = {Dynamical versions of {H}ardy's uncertainty principle: a
              survey},
   JOURNAL = {Bull. Amer. Math. Soc. (N.S.)},
  FJOURNAL = {American Mathematical Society. Bulletin. New Series},
    VOLUME = {58},
      YEAR = {2021},
    NUMBER = {3},
     PAGES = {357--375},
      ISSN = {0273-0979,1088-9485},
   MRCLASS = {42A38 (35B05)},
  MRNUMBER = {4273105},
MRREVIEWER = {Francisco\ J.\ Mendoza Torres},
       DOI = {10.1090/bull/1729},
       URL = {https://doi.org/10.1090/bull/1729},
}

@article {FS,
    AUTHOR = {Folland, Gerald B. and Sitaram, Alladi},
     TITLE = {The uncertainty principle: a mathematical survey},
   JOURNAL = {J. Fourier Anal. Appl.},
  FJOURNAL = {The Journal of Fourier Analysis and Applications},
    VOLUME = {3},
      YEAR = {1997},
    NUMBER = {3},
     PAGES = {207--238},
      ISSN = {1069-5869,1531-5851},
   MRCLASS = {42B10 (22E30 43A30 81S05 81S30)},
  MRNUMBER = {1448337},
MRREVIEWER = {Michael\ Cowling},
       DOI = {10.1007/BF02649110},
       URL = {https://doi.org/10.1007/BF02649110},
}

@book {deGosQHM,
    AUTHOR = {de Gosson, Maurice A.},
     TITLE = {Quantum Harmonic Analysis - an Introduction},
    SERIES = {Advances in Analysis and Geometry},
    VOLUME = {4},
 PUBLISHER = {De Gruyter, Berlin},
      YEAR = {2021},
     PAGES = {xvii+220},
      ISBN = {978-3-11-072261-1; 978-3-11-072277-2; 978-3-11-072290-1},
   MRCLASS = {81-02 (32A50 35Q40 42B35 42B37 47N50 81P40 81S30)},
  MRNUMBER = {4487418},
MRREVIEWER = {H.\ Hogreve},
       DOI = {10.1515/9783110722772},
       URL = {https://doi.org/10.1515/9783110722772},
}

@book {corderobook,
    AUTHOR = {Cordero, Elena and Rodino, Luigi},
     TITLE = {Time-Frequency Analysis of Operators},
    SERIES = {De Gruyter Studies in Mathematics},
    VOLUME = {75},
 PUBLISHER = {De Gruyter, Berlin},
      YEAR = {2020},
     PAGES = {xiv+442},
      ISBN = {978-3-11-053245-6; 978-3-11-053035-3; 978-3-11-053060-5},
   MRCLASS = {42B35 (35B40 35B65 35Q55 42B10 42B37 46N20 47G30 81S30)},
  MRNUMBER = {4201879},
MRREVIEWER = {Ondrej\ Hutn\'{\i}k},
       DOI = {10.1515/9783110532456},
       URL = {https://doi.org/10.1515/9783110532456},
}

@book {Gos11,
    AUTHOR = {de Gosson, Maurice A.},
     TITLE = {Symplectic Methods in Harmonic analysis and in Mathematical Physics},
    SERIES = {Pseudo-Differential Operators. Theory and Applications},
    VOLUME = {7},
 PUBLISHER = {Birkh\"{a}user/Springer Basel AG, Basel},
      YEAR = {2011},
     PAGES = {xxiv+337},
      ISBN = {978-3-7643-9991-7},
   MRCLASS = {53Dxx (35Sxx 37J05 42B37 47G30 81S30)},
  MRNUMBER = {2827662},
MRREVIEWER = {Miroslav\ Engli\v{s}},
       DOI = {10.1007/978-3-7643-9992-4},
       URL = {https://doi.org/10.1007/978-3-7643-9992-4},
}

@article{ShafkulovskaGrochenig2025b,
	author  = {Gr{\"o}chenig, Karlheinz and Shafkulovska, Iryna},
	title   = {More Uncertainty Principles for Metaplectic Time-Frequency Representations},
	journal = {arXiv preprint},
	year    = {2025},
	eprint  = {2503.13324},
	archivePrefix = {arXiv}
}

@article {GrochenigShafkulovska2025,
    AUTHOR = {Gr\"{o}chenig, Karlheinz and Shafkulovska, Irina},
     TITLE = {Benedicks-type uncertainty principle for metaplectic
              time-frequency representations},
   JOURNAL = {J. Anal. Math.},
  FJOURNAL = {Journal d'Analyse Math\'{e}matique},
    VOLUME = {158},
      YEAR = {2026},
    NUMBER = {2},
     PAGES = {573--603},
      ISSN = {0021-7670,1565-8538},
   MRCLASS = {42B10 (47G30)},
  MRNUMBER = {5014027},
       DOI = {10.1007/s11854-025-0427-7},
       URL = {https://doi.org/10.1007/s11854-025-0427-7},
}

@article {Hardy33,
    AUTHOR = {Hardy, G. H.},
     TITLE = {A {T}heorem {C}oncerning {F}ourier {T}ransforms},
   JOURNAL = {J. London Math. Soc.},
  FJOURNAL = {The Journal of the London Mathematical Society},
    VOLUME = {8},
      YEAR = {1933},
    NUMBER = {3},
     PAGES = {227--231},
      ISSN = {0024-6107,1469-7750},
   MRCLASS = {DML},
  MRNUMBER = {1574130},
       DOI = {10.1112/jlms/s1-8.3.227},
       URL = {https://doi.org/10.1112/jlms/s1-8.3.227},
}

@article {H,
    AUTHOR = {H\"{o}rmander, Lars},
     TITLE = {A uniqueness theorem of {B}eurling for {F}ourier transform
              pairs},
   JOURNAL = {Ark. Mat.},
  FJOURNAL = {Arkiv f\"{o}r Matematik},
    VOLUME = {29},
      YEAR = {1991},
    NUMBER = {2},
     PAGES = {237--240},
      ISSN = {0004-2080,1871-2487},
   MRCLASS = {42A38 (01A60)},
  MRNUMBER = {1150375},
MRREVIEWER = {Colin\ C.\ Graham},
       DOI = {10.1007/BF02384339},
       URL = {https://doi.org/10.1007/BF02384339},
}

@article{Liu,
AUTHOR = {Jing, Rui and Liu, Bei and Li, Rui and Liu, Rui},
TITLE = {The $N$-Dimensional Uncertainty Principle for the Free Metaplectic Transformation},
JOURNAL = {Mathematics},
VOLUME = {8},
YEAR = {2020},
NUMBER = {10},
ARTICLE-NUMBER = {1685},
URL = {https://www.mdpi.com/2227-7390/8/10/1685},
ISSN = {2227-7390},
DOI = {10.3390/math8101685}
}

@article{Morgan,
author = {Morgan, G. W.},
title = {A Note on {F}ourier Transforms},
journal = {Journal of the London Mathematical Society},
volume = {s1-9},
number = {3},
pages = {187-192},
doi = {https://doi.org/10.1112/jlms/s1-9.3.187},
url = {https://londmathsoc.onlinelibrary.wiley.com/doi/abs/10.1112/jlms/s1-9.3.187},
eprint = {https://londmathsoc.onlinelibrary.wiley.com/doi/pdf/10.1112/jlms/s1-9.3.187},
year = {1934}
}

@article{ST2024,
  title={Subcritical {F}ourier uncertainty principles},
  author={Saucedo, Miquel and Tikhonov, Sergey},
  journal={arXiv preprint arXiv:2404.07375},
  year={2024}
}

@article {WW21,
    AUTHOR = {Wigderson, Avi and Wigderson, Yuval},
     TITLE = {The uncertainty principle: variations on a theme},
   JOURNAL = {Bull. Amer. Math. Soc. (N.S.)},
  FJOURNAL = {American Mathematical Society. Bulletin. New Series},
    VOLUME = {58},
      YEAR = {2021},
    NUMBER = {2},
     PAGES = {225--261},
      ISSN = {0273-0979,1088-9485},
   MRCLASS = {43A25 (20C15 81S07 94A12)},
  MRNUMBER = {4229152},
       DOI = {10.1090/bull/1715},
       URL = {https://doi.org/10.1090/bull/1715},
}

@article {Z1,
    AUTHOR = {Zhang, Zhichao},
     TITLE = {Uncertainty principle for free metaplectic transformation},
   JOURNAL = {J. Fourier Anal. Appl.},
  FJOURNAL = {The Journal of Fourier Analysis and Applications},
    VOLUME = {29},
      YEAR = {2023},
    NUMBER = {6},
     PAGES = {Paper No. 71, 33},
      ISSN = {1069-5869,1531-5851},
   MRCLASS = {15A18 (15A42 15B48 42B10)},
  MRNUMBER = {4662436},
MRREVIEWER = {Mark\ Kempton},
       DOI = {10.1007/s00041-023-10052-0},
       URL = {https://doi.org/10.1007/s00041-023-10052-0},
}

@article {ZH,
    AUTHOR = {Zhang, Zhichao and Zhu, Zhicheng and Li, Dong and He, Yangfan},
     TITLE = {Free metaplectic {W}igner distribution: definition and
              {H}eisenberg's uncertainty principles},
   JOURNAL = {IEEE Trans. Inform. Theory},
  FJOURNAL = {Institute of Electrical and Electronics Engineers.
              Transactions on Information Theory},
    VOLUME = {69},
      YEAR = {2023},
    NUMBER = {10},
     PAGES = {6787--6810},
      ISSN = {0018-9448,1557-9654},
   MRCLASS = {42B10},
  MRNUMBER = {4650334},
MRREVIEWER = {Luigi\ Rodino},
       DOI = {10.1109/tit.2023.3287990},
       URL = {https://doi.org/10.1109/tit.2023.3287990},
}

@article {MO2002,
    AUTHOR = {ter Morsche, Hennie and Oonincx, Patrick J.},
     TITLE = {On the integral representations for metaplectic operators},
   JOURNAL = {J. Fourier Anal. Appl.},
  FJOURNAL = {The Journal of Fourier Analysis and Applications},
    VOLUME = {8},
      YEAR = {2002},
    NUMBER = {3},
     PAGES = {245--257},
      ISSN = {1069-5869,1531-5851},
   MRCLASS = {42A38 (42B10)},
  MRNUMBER = {1906251},
       DOI = {10.1007/s00041-002-0011-8},
       URL = {https://doi.org/10.1007/s00041-002-0011-8},
}

@article {Lerner,
    AUTHOR = {Lerner, Nicolas},
     TITLE = {On the uncertainty principle for metaplectic transformations},
   JOURNAL = {J. Funct. Anal.},
  FJOURNAL = {Journal of Functional Analysis},
    VOLUME = {289},
      YEAR = {2025},
    NUMBER = {5},
     PAGES = {Paper No. 110997, 59},
      ISSN = {0022-1236,1096-0783},
   MRCLASS = {81S30 (37J11 47G30 81S07)},
  MRNUMBER = {4896809},
MRREVIEWER = {Yuval\ Wigderson},
       DOI = {10.1016/j.jfa.2025.110997},
       URL = {https://doi.org/10.1016/j.jfa.2025.110997},
}

@article {LPDM2026,
    AUTHOR = {Liang, Ping and Dang, Pei and Mai, Weixiong},
     TITLE = {Uncertainty principles for free metaplectic transformation and
              associated metaplectic operators},
   JOURNAL = {J. Geom. Anal.},
  FJOURNAL = {Journal of Geometric Analysis},
    VOLUME = {36},
      YEAR = {2026},
    NUMBER = {6},
     PAGES = {Paper No. 192, 52},
      ISSN = {1050-6926,1559-002X},
   MRCLASS = {42A38 (81S10 94A12)},
  MRNUMBER = {5068419},
       DOI = {10.1007/s12220-026-02442-9},
       URL = {https://doi.org/10.1007/s12220-026-02442-9},
}

\end{document}